\documentclass[a4paper,twoside,10pt]{article}

\usepackage[english]{babel}
\usepackage[latin1]{inputenc}
\usepackage{lmodern}
\usepackage[T1]{fontenc}
\usepackage{indentfirst}
\usepackage{authblk}
\usepackage[a4paper,top=3cm,bottom=3cm,left=3cm,right=3cm,%
bindingoffset=5mm]{geometry}
\usepackage{lmodern}
\usepackage[T1]{fontenc}
\usepackage{lettrine}
\usepackage{booktabs}
\usepackage{epsfig}
\usepackage{epstopdf}
\usepackage{emp}
\usepackage{amsthm}
\usepackage{amsmath,amssymb,amsthm}
\usepackage{braket}
\usepackage{chngpage}
\usepackage{cases}
\usepackage{graphicx}
\usepackage{tabularx}
\usepackage{multirow}

\usepackage{epstopdf}
\newcommand{\numberset}{\mathbb}

\newcommand{\R}{\numberset{R}}

\newcommand{\K}{\numberset{K}}
\newcommand{\M}{\numberset{M}}

\newcommand{\gd}{\mathcal{GRAD}}
\newcommand{\did}{\mathcal{DIV}}
\renewcommand{\epsilon}{\varepsilon}
\renewcommand{\theta}{\vartheta}
\renewcommand{\rho}{\varrho}
\renewcommand{\phi}{\varphi}
{\left\lbrace\begin{array}{@{}l@{}}}%
{\end{array}\right.}
\theoremstyle{definition}

\theoremstyle{remark}
\newtheorem{test}{Test}[section]
\theoremstyle{remark}
\newtheorem{osservazione}{Remark}[section]

\theoremstyle{plain}
\newtheorem{teorema}{Theorem}[section]

\newtheorem{lemma}{Lemma}[section]

\usepackage{setspace}

\author[1]{L. Beir\~ao da Veiga \thanks{lourenco.beirao@unimib.it}}
\author[2]{L. Lopez  \thanks{luciano.lopez@uniba.it}}
\author[1]{G. Vacca  \thanks{giuseppe.vacca@unimib.it}}

\affil[1]{Dipartimento di Matematica e Applicazioni,  Universit\`a degli Studi di Milan-Bicocca, Via Roberto Cozzi, 55 - 20125}
\affil[2]{Dipartimento di Matematica,  Universit\`a degli Studi di Bari, Via Edoardo Orabona, 4 - 70125 Bari}

\title{\textbf{Mimetic Finite Difference methods for Hamiltonian wave equations in 2D}}
\date{\today}
\begin{document}
\maketitle

\begin{abstract}
In this paper we consider the numerical solution of the Hamiltonian wave equation in two spatial dimensions. We construct a two step procedure in which we first discretize the space by the Mimetic Finite Difference (MFD) method and then we employ a standard  symplectic scheme to integrate the semi-discrete
Hamiltonian system derived. The main characteristic of the MFD methods, when applied to stationary problems, is to mimic important properties of the continuous system. This approach yields a full numerical procedure suitable to integrate
Hamiltonian problems. A complete theoretical analysis of the method and some numerical simulations are developed in the paper.
\end{abstract}

%\blfootnote{Work performed while the author Giuseppe Vacca was in visit to the Department of Mathematics of the University of Pavia whose support is gratefully acknowledged.}
% -----------------------------------------------------------------------------------------------------------------------------
% -----------------------------------------------------------------------------------------------------------------------------
\section{Introduction}

Because of the symplectic structures, Hamiltonian partial differential equations (PDEs) are used to give a mathematical representation of many physical systems and are of interest to various applicative fields, see for instance quantum field theory, meteorology, nonlinear optics, weather forecast.

An important requirement that any numerical method for Hamiltonian PDEs has to satisfy is the preservation of the intrinsic geometric properties of the original continuous problem. In particular, the numerical procedure should preserve the symplectic  structure of the Hamiltonian system during numerical simulations.  A standard procedure to derive a suitable method for an infinite-dimensional Hamiltonian PDE consists into two steps: in the first one the system is discretized in space in order to obtain a finite-dimensional Hamiltonian system, and then the semi-discretized system is solved in time by a symplectic integrator \cite{MR1904823, MR2657217, MR2132573, gawlik2011, gawlik2014, demoures2015}. There exists also a recent approach in which the space and time are considered on equal footing, this approach requires a multi-symplectic formulation of the system and leads to the multi-symplectic numerical schemes for the numerical solution of the PDEs (see \cite{MR1854689, MR2220764, MR2222808,  marsden1998}).

The effectiveness of this approach is ensured by the property that the derived semi-discrete system is a finite-dimensional Hamiltonian
system of ordinary differential equations (ODEs).
The space discretization of a Hamiltonian system is usually performed by one of the following techniques: finite difference methods, finite element methods, spectral methods, pseudospectral methods, Fourier expansion,  wavelet based methods (see for instance \cite{MR1472194, MR1997061,  MR2211034, MR2425152, MR2586202}). However, these semi-discretization  approaches could become very expensive or could not be applicable when the space dimension $d$ is greater than $d=1$.

Instead, in this paper we  consider the Mimetic Finite Difference (MFD) method to approximate the continuous problem combined with a standard symplectic integration in time to integrate the derived semi-discrete Hamiltonian system.

The main results about MFD methods, for stationary problems, can be found in the recent book \cite{BLM11book} and papers \cite{MR3133437, MR2192322} where, in particular, the theoretical framework of the mimetic spaces and the discretization of the operators are introduced. Significative applications of  MFD methods may be found for instance in \cite{MR2339994, MR2512497, MR2534128, MR2568590, MR2655949, MR2737630}. Among the first publication in this field it is worth mentioning \cite{MR1383592, MR1426277} where a first approach to mimetic discretization of the continuous operators can be found and the fundamental papers \cite{MR2168945, MR3133438} where the modern approach to MFD was introduced. A generalization of the MFD methods has been recently proposed, the \textit{virtual element methods} (VEMs); we cite \cite{MR3073346,  MR2997471, VEM-helmholtz, vaccabis, vaccahyper, vaccadivfree} as a very short representative list.

Recently in \cite{vacca}, MDF methods has been applied to the space discretization of PDEs of parabolic type in two dimension, showing how this technique preserves invariants of the solution better than classical space discretizations such as finite difference methods.

The main characteristic of the MFD methods 
%, when applied to stationary problems,
is to mimic important properties of the continuous system, e.g., conservation laws, symmetry and positivity
of the solutions, and the most important properties of the continuous differential operators,
including duality and self-adjointness relations. Furthermore  MFD methods can be applied
for general polygonal and polyhedral meshes of the space domain instead of more standard triangular/quadrilateral grids.

The main novelty of this paper is the use of  MFD methods for the space discretization of the nonlinear wave equation in 2D coupled with a standard symplectic method (the implicit midpoint scheme) for the time integration. We derive a full numerical discretization procedure which will exploit the conservative properties of the MFD approach associated to the symplectic features of the time integrator.  We  show that the {\em mimetic} semi-discrete Hamiltonian  is preserved in time and we  derive the conservation law for the {\em mimetic} semi-discrete energy. Furthermore we give a bound for the conservation of the full discretized Hamiltonian and for the conservation of the full discretized energy. We also prove the convergence of the semi-discrete and fully discrete solutions to the solution of the original problem

The paper is organized in the following way. In Section \ref{sec:2} we recall the basic elements of the MFD approach.
In  Section \ref{sec:3} we recall the mathematical form of the Hamiltonian PDE we wish to study. In Section \ref{sec:4} we apply the MFD method to the continuous problem and we give a result of the convergence of the semi-discrete solution to the continuous solution of the original problem; we define the semi-discrete Hamiltonian and energy density, show their conservation laws. In  Section \ref{sec:5} we discretize the semi-discrete system by using a symplectic time integrator, the implicit midpoint rule, of the second order in time. We will prove the convergence of the full discrete numerical solution by providing an error estimate of the second order in space and time. Hence we give a result about the conservation of the discrete Hamiltonian and of the discrete energy of the system. Section \ref{sec:6} is devoted to show some numerical results.

% --------------------------------------------------------------------------------------------------------------------------
% --------------------------------------------------------------------------------------------------------------------------

\section{Background on Mimetic Finite Differences Methods}
\label{sec:2}
In this section, for ease of reading,  we recall the basic concepts and notations on MFD methods which will be used to discretize PDEs in the spatial domain $\Omega \subseteq \R^2$ where we assume $\Omega$ bounded polygon. 
For more details on this subject we  refer the interested reader to the recent book \cite{BLM11book} or  to the papers \cite{MR2192322, MR3133437, vacca}.  
Let $\omega$ a measurable subset of the domain $\Omega$  and let $\K \in (L^{\infty}(\Omega))^{2 \times 2}$ a full symmetric positive definite tensor.
By making use of standard notation, we consider  the following scalar products:
\begin{align}
\label{eq:psh0}
(u, \,  v)_{L^2(\omega)} &:= \int_{\omega} u \, v \, {\rm d}x \qquad \text{for all $u,\, v \in L^2(\omega)$},  \\
\label{eq:psH}
( \boldsymbol{\omega}, \,  \boldsymbol{\sigma})_{\K, \omega} &:= \int_{\omega} \K^{-1} \boldsymbol{\omega} \cdot  \boldsymbol{\sigma} \, {\rm d}x \qquad \text{for all $\boldsymbol{\omega}, \,  \boldsymbol{\sigma} \in (L^2(\omega))^2$}. 
\end{align}
It is clear that, in the sense of distribution,
\begin{displaymath}
(\K \nabla u, \, \boldsymbol{\sigma})_{\K, \Omega} = - (u, \, {\rm div} \boldsymbol{\sigma})_{L^2(\Omega)} \qquad \text{for all $u \in L^2(\Omega)$, $\boldsymbol{\sigma} \in H({\rm div}, \Omega)$}
\end{displaymath}
%Let us define the operators  ${\di}\colon H({\rm div} , \Omega) \to L^2(\Omega)$ and ${\g}\colon H^1_0(\Omega) \to L^2(\Omega; \R^d)$ with
%\[
%{\di}\,\boldsymbol{\omega} := \nabla \cdot \boldsymbol{\omega}\quad    {\rm in\ }   \Omega\ ,  \quad \forall \boldsymbol{\omega} \in H({\rm div} ,\Omega); \qquad {\g} \, u := \K \nabla u\ , \qquad \forall u  \in H_0^1(\Omega)\ .
%\end{displaymath}
%Thus, using Green's formula, we easily get the duality relation:
thus we get the duality relation with respect to the scalar product  \eqref{eq:psh0} and  \eqref{eq:psH}
\begin{equation}
\label{eq:dualit1}
\K \nabla = -({\rm div})^*.
\end{equation}

%We suppose that system \eqref{eq:ondesym} is posed in a bounded polyhedral (polygonal) domain $\Omega \subseteq \R^d$  with a Lipschitz continuous boundary.

 Let $\mathcal{T}_h$ be an unstructured mesh of $\Omega$ into nonoverlapping simply-connected polygons with flat faces, where
\begin{displaymath}
h := \sup_{c \in \mathcal{T}_h} \text{diameter$(c)$}.
\end{displaymath}
Let $\mathcal{E}_h$ be the set of edges of the polygons in $\mathcal{T}_h$.  We use the following notations for the mesh objects: 
 $c \in \mathcal{T}_h$ denotes a general cell in the mesh with measure $|c|$ and centroid $\mathbf{x}_c$; 
$f \in \mathcal{E}_h$ denotes a general edge of the cell $c$ with measure $|f|$ and centroid $\mathbf{x}_f$;
$\mathbf{n}_f$ indicates the unit normal vector to the edge $f$ with preassigned direction;
 $\alpha_{c,f} = \pm 1$  represents the mutual orientation of the vector $\mathbf{n}_f$ and the outward normal vector to $f$ with respect to the cell $c$. 

Moreover, let $\mathcal{Y}_h = \mathcal{T}_h, \mathcal{E}_h$, and let $\sigma = c, f$, then we denote with $\mathcal{Y}_h(\sigma)$ the subset of $\mathcal{Y}_h$ of all the elements that are related with $\sigma$, and we indicate with $|\mathcal{Y}_h(\sigma)|$ the cardinality of this set. For example $\mathcal{T}_h(f)$ denotes all cells sharing face $f$ and $\mathcal{E}_h(c)$ denotes all faces forming the boundary of cell $c$. 

In the following we take on the element $c \in \mathcal{T}_h$ the shape regularity assumptions listed, for instance, in \cite{BLM11book, MR2192322}. A possibility is to assume that for all $h$, each element $c$ in $\mathcal{T}_h$ satisfies: 
\begin{itemize}
\item[\textbf{(M1)}] $c$ is star-shaped with respect to a ball of radius greater then $\gamma \, h_c$,
\item[\textbf{(M2)}] any two vertexes in  $c$ are at least $\sigma \, h_K$ apart, 
\end{itemize}
where $h_c$ is the diameter of $c$. The constants $\gamma$ and $\sigma$ are positive and uniform with respect to the mesh family.

The mesh objects will define the degrees of freedom of the discrete system, that is these will define the space of the \textbf{discrete pressures} and  \textbf{discrete fluxes}.

Let
\begin{displaymath}
N_c := |\mathcal{T}_h|, \qquad N_f := |\mathcal{E}_h|,  \qquad N^* := \max_{c}|\mathcal{E}_h(c)|.
\end{displaymath}
Let $\mathcal{C}_h$ be the set of the  pressures that are piecewise constant on $\mathcal{T}_h$, i.e.
\begin{displaymath}
\mathcal{C}_h := \set{u \in L^2(\Omega) | u_{|c} = {\rm const}, \quad \forall c \in \mathcal{T}_h}.
\end{displaymath}
Given a pressure $u \in L^2(\Omega)$,  we define  the \textbf{interpolant discrete pressure} $u^I\in \mathcal{C}_h$  with
\begin{displaymath}
u_{|c}^I = \frac{1}{|c|} \int_c u \, {\rm d}c, \qquad \text{for all c $\in \mathcal{T}_h$.}
\end{displaymath}
The space $\mathcal{F}_h$  of the discrete velocities  is defined as follows. For all edge $f\in \mathcal{E}_h$ we associate a real number $\omega_f$ and we denote with $\boldsymbol{\omega}_h$ the vector with components given by the collection of all the $\{ \omega_f\}_{f \in \mathcal{E}_h}$. The symbol $\mathcal{F}_h$ will represent  the vector space of all  $\boldsymbol{\omega}_h$. Let $\boldsymbol{\omega} \in H({\rm div} ,c)$ a vector function, and let us assume that all face-integrals 
\begin{displaymath}
\int_f \boldsymbol{\omega}\cdot \mathbf{n}_f\, {\rm d}S, \qquad \text{for all $f \in \mathcal{E}_h$}
\end{displaymath}
exist. Then the \textbf{interpolant discrete flux}  of $\boldsymbol{\omega}$ in the space  $\mathcal{F}_h$ is defined by $\boldsymbol{\omega}^I:= (\omega_f)_{f \in \mathcal{E}_h}$ with
\begin{displaymath}
\omega_f  = \frac{1}{|f|} \int_f \boldsymbol{\omega}\cdot \mathbf{n}_f\, {\rm d}S, \qquad \text{for all $f \in \mathcal{E}_h$.}
\end{displaymath}

\begin{osservazione}
\label{oss:0}
The discrete spaces $\mathcal{C}_h$, $\mathcal{F}_h$ and the interpolation operators are defined starting from the degrees of freedom:
\begin{itemize}
\item $\frac{1}{|c|} \int_c u \, {\rm d}c$, \qquad \text{for all $c \in \mathcal{T}_h$, and $u \in L^2(\Omega)$,} 
\item $\frac{1}{|f|} \int_f \boldsymbol{\omega} \cdot \mathbf{n}_f \, {\rm d}S$, \qquad \text{for all $f \in \mathcal{E}_h(c)$, and $\boldsymbol{\omega} \in H({\rm div} , c)$.}
\end{itemize}
\end{osservazione}

\begin{osservazione}
\label{oss:1}
There are obvious correspondences:
\begin{displaymath}
\mathcal{C}_h \cong \R^{N_c}  \qquad u \mapsto (u_{c})_{c \in \mathcal{T}_h}, \qquad \text{and} \qquad
\mathcal{F}_h \cong \R^{N_f}  \qquad \boldsymbol{\omega} \mapsto (\omega_{f})_{f \in \mathcal{E}_h}.
\end{displaymath}
With a slight abuse of notation we can refer to a function in the discrete functional spaces as a vector and vice versa. 
\end{osservazione}

The definition of the mimetic scheme   carries on with the discretisation of the differential operators. Let $\boldsymbol{\omega} \in H({\rm div} , c)$ with $c \in \mathcal{T}_h$, then the Divergence Theorem states that
\begin{displaymath}
\int_{c} {\rm div} \boldsymbol{\omega} \, {\rm d}x = \int_{\partial c} \boldsymbol{\omega} \cdot \mathbf{n} \, {\rm d}S\ ,
\end{displaymath} 
where $\mathbf{n}$ is the unit outward normal to $\partial c$. Therefore, the continuous operator ${\rm div}$ admits the immediate discretisation  
${\did} \colon \mathcal{F}_h \to \mathcal{C}_h, $ with
\begin{displaymath}
({\did}\boldsymbol{\omega}_h)_c = \frac{1}{|c|} \sum_{f \in \mathcal{E}_h(c)} \alpha_{c,f} |f| \omega_f \quad \forall  \, \boldsymbol{\omega}_h\in \mathcal{F}_h.
\end{displaymath}
The operator ${\did}$ is called \textbf{discrete primary operator}.

The next step in the construction of the MFD method is the definition of suitable inner products on the discrete functional spaces  $\mathcal{C}_h$ and $\mathcal{F}_h$ that allow to construct the derived operators imposing the duality relations for the discrete operators. 

We assume, for the moment, the following scalar products on the vector spaces  $\mathcal{C}_h$ and  $\mathcal{F}_h$:
\begin{gather}
\label{eq:proc}  
 [u_h, v_h]_{\mathcal{C}_h}                   := u_h^T \M_{\mathcal{C}_h}v_h                \qquad \text{for all $u_h, v _h\in \mathcal{C}_h$,} \\
\label{eq:prof}
 [\boldsymbol{\omega}_h, \boldsymbol{\sigma}_h]_{\mathcal{F}_h} := \boldsymbol{\omega}_h^T \M_{\mathcal{F}_h}\boldsymbol{\sigma}_h \qquad \text{for all $\boldsymbol{\omega}_h, \boldsymbol{\sigma}_h \in \mathcal{F}_h$,} 
\end{gather}
where $\M_{\mathcal{C}_h}\in\R^{N_c \times N_c}$,  $\M_{\mathcal{F}_h} \in \R^{N_f \times N_f}$  are suitable symmetric positive definite matrices. These matrices are locally constructed in such a way, on each cell, the corresponding  local discrete inner products  have to ``mimic'' the scalar products defined in   \eqref{eq:psh0}) and   \eqref{eq:psH}. Therefore we would like that
\begin{gather*}
 [u_{h,c}, v_{h,c}]_{{\mathcal{C}}_{h,c}} =: (u_{h,c})^T \M_{\mathcal{C}_{h,c}}v_{h,c} \approx  (u_h, \, v_h)_{L^2(c)}, \qquad \text{for all $u_h,v_h \in \mathcal{C}_h$,}  \\
  [\boldsymbol{\omega}_{h,c}, \boldsymbol{\sigma}_{h,c}]_{{\mathcal{F}}_{h,c}} =: (\boldsymbol{\omega}_{h,c})^T \M_{\mathcal{F}_{h,c}}\boldsymbol{\sigma}_{h,c} \approx  (\boldsymbol{\omega}_h\, , \boldsymbol{\sigma}_h)_{\K, c}, \qquad \text{for all $\boldsymbol{\omega}_h, \boldsymbol{\sigma}_h \in \mathcal{F}_h$,} 
\end{gather*}
where, in general, with the notation $r_{h,c}$ we denote the vector with the degrees of freedom of the function $r$ relative to the cell $c$.

As regards the first local inner products, we observe that  the vector $r_{h,c}$ has a single component, representing the (constant) value of $r_h$ in the cell $c$. Then the only possible quadrature formula  is
\begin{displaymath}
[u_{h,c}, v_{h,c}]_{{\mathcal{C}}_{h,c}} = (u_{h,c})^T\M_{{\mathcal{C}}_{h,c}}v_{h,c} = |c|u_{c}\,v_{c},
\end{displaymath}
therefore $\M_{{\mathcal{C}}_{h,c}} = |c|$ and $\M_{{\mathcal{C}_h}} :=  {\rm diag}   (|c_1|, \dots, |c_{N_c}|)$. It is clear that the discrete inner products gives the exact value of the continuous one whenever $u_h,v_h \in \mathcal{C}_h$.

The definition of the local scalar product for the fluxes requires a different approach. The key idea is to define suitable  consistency and stability constraints  in order to introduce algebraic conditions on the elements of the matrix $\M_{{\mathcal {F}}_{h,c}}$. Without spelling things out, we requires that the following properties are satisfied
\begin{itemize}
\item consistency: let  $\boldsymbol{\omega}, \boldsymbol{\sigma}$ two vector fields and let $\boldsymbol{\omega}_h, \boldsymbol{\sigma}_h \in \mathcal{F}_h$ their interpolant functions. If $\boldsymbol{\omega}$ is constant in $c$ and for each edge in $f \in \mathcal{E}_h(c)$,  $\boldsymbol{\sigma} \cdot \mathbf{n}_f$ is constant, then 
\begin{displaymath}
[\boldsymbol{\omega}_{h,c}, \boldsymbol{\sigma}_{h,c}]_{{\mathcal{F}}_{h,c}} = \int_c \K^{-1} \boldsymbol{\omega} \cdot\boldsymbol{\sigma} \, {\rm d}c;
\end{displaymath}
\item stability: there exist two positive $h$-independent constants  $C_*$ and $C^*$  such that
\begin{displaymath}
C_* |c| (\mathbf{\omega}_{h,c})^T \mathbf{\omega}_{h,c} \leq (\mathbf{\omega}_{h,c}^T) \M_{\mathcal{F}_ {h,c}} \mathbf{\omega}_{h,c} \leq C^* |c| (\mathbf{\omega}_{h,c}^T)  \mathbf{\omega}_{h,c} \qquad \forall \mathbf{\omega}_{h,c} \in \mathcal{F}_{h}.
\end{displaymath}
\end{itemize}

The last preliminary step in the construction of the MFD method is the definition of the \textbf{derived discrete operators}, which are obtained through a duality relation from the primary operators. Let us consider the spaces $\mathcal{C}_h$,  $\mathcal{F}_h$ equipped respectively with the scalar products  \eqref{eq:proc},   \eqref{eq:prof}. From continuous duality relations  \eqref{eq:dualit1}, we can introduce the discrete operator  
\begin{displaymath}
{\gd} \colon \mathcal{C}_h \to \mathcal{F}_h
\end{displaymath}
and impose the duality relation:
\begin{displaymath}
[ \boldsymbol{\omega}_h, {\gd} \, u_h]_{\mathcal{F}_h} = -[{\did} \, \boldsymbol{\omega}_h, u_h]_{\mathcal{C}_h}\Leftrightarrow  \boldsymbol{\omega}_h^T \M_{\mathcal{F}_h} {\gd}\,u_h = -\boldsymbol{\omega}_h^T {\did}^T \M_{\mathcal{C}_h} u_h\ , 
\end{displaymath}
for all $\boldsymbol{\omega}_h \in \mathcal{F}_h, u_h \in \mathcal{C}_h$, from which it follows that
\begin{displaymath}
{\gd} := - \M_{\mathcal{F}_h}^{-1} {\did}^T \M_{\mathcal{C}_h}.
\end{displaymath}
Finally we can introduce the discrete counterpart of the continuous operator ${\rm div}\, \K \nabla$, by defining the operator
\begin{displaymath}
\Delta_h \colon \mathcal{C}_h \to \mathcal{C}_h
\end{displaymath} 
given by
\begin{equation}
\label{eq:mfddelta}
\Delta_h := {\did} \, {\gd}.
\end{equation}
% -----------------------------------------------------------------------------------------------------------------------------

% -----------------------------------------------------------------------------------------------------------------------------
% -----------------------------------------------------------------------------------------------------------------------------
\section{The continuous problem}
\label{sec:3}

Let $\Omega \subseteq \R^2$ be a bounded  polygon and let us consider the \textbf{nonlinear wave equation} with  homogeneous boundary value problem
\begin{equation}
\label{eq:onde}
\left \{
\begin{aligned}  
& u_{tt}(x,t)    = {\rm div} \K \nabla u(x,t) - f'(u(x,t))   & \quad  \text{in $\Omega \times (0, T)$}  \\             
& u(x,0) = u_0(x)\ , \quad u_t(x,0) = v_0(x)                   &\quad  \text{in $\Omega$}               \\
& u(x,t) = 0                                                    & \quad \text{on $\partial \Omega \times (0,T)$} 
\end{aligned}       
\right .
\end{equation}
where $\K \in (W^{1,\infty})^{2 \times 2}$  is a full symmetric positive definite tensor, and the source term $f'$ is the derivative of a smooth function $f \colon \R \to \R$.
We would observe that no particularly restrictive assumptions on $f'$ are required, for instance $f'$ in the sine-Gordon equation or the ones of polynomial type with respect to $u$  
%(see for instance \cite{gauckler2015}) 
may be considered. For seek of simplicity we consider in the proof $f'$ global Lipschitz, however the convergence results are still valid for $f'$ local Lipschitz (see  Remark \ref{oss:local_lip})

%The ODE $z_t = f(z)$ with $z \colon \Omega \times (0,T) \to \R^{2m}$ is said \textbf{Hamiltonian system} if there exists a functional %$\mathcal{H}$ such that the ODE can be written in the form 
%\[
%z_t = \mathcal{J}_m\delta_z \mathcal{H}(z)
%\end{displaymath}
%where $\mathcal{J}_m = \bigl( \begin{smallmatrix}& 0 & I_m \\ & -I_m & 0  \end{smallmatrix}\bigr)$ is the canonical matrix and %$\delta_z\mathcal{H}(z)$ denotes the functional derivative of the %functional $\mathcal{H}$ defined by
%\[
%\int_{\Omega} (\delta_z\mathcal{H}(z)\, \sigma_z) \, {\rm d}x = \lim_{\epsilon \to 0} \frac{\mathcal{H}(z + \epsilon \, \sigma_z) - %\mathcal{H}(z)}{\epsilon}
%\end{displaymath}
%for all variation $\sigma_z \colon \Omega \to \R^2$ with homogeneous boundary conditions.
 \eqref{eq:onde} admits the equivalent formulation
\begin{equation}
\label{eq:ondesym}
\left \{
\begin{aligned}  
& u_t (x,t) = v(x,t)                &  \text{in $\Omega \times (0, T)$} \\  
& v_t (x,t) = {\rm div} \K \nabla u(x,t) - f'(u(x,t))   &  \text{in $\Omega \times (0, T)$} 
\end{aligned}       
\right.
\end{equation}
where the initial and boundary conditions are given by
\begin{gather*}
u(x,0) = u_0(x), \quad v(x,0) = v_0(x)                    \quad \text{in $\Omega$}             \qquad
u(x,t) = 0, \quad v(x,t) = 0  \ ,                          \quad  \text{on $\partial \Omega \times (0,T)\ .$}
\end{gather*}
 \eqref{eq:ondesym} is said \textbf{Hamiltonian formulation} of  \eqref{eq:onde} for which the \textbf{Hamiltonian}
\begin{equation}
\label{eq:funzionale}
\mathcal{H}[u,v] := \int_{\Omega} \left( \frac{1}{2} v^2 + \frac{1}{2} \nabla u \cdot \K  \nabla u + f(u) \right){\rm d}x
\end{equation}
is invariant with respect to time $t$ along the solution, that is
%We can observe that along the solutions $u$, $v$ of  problem \eqref{eq:ondesym}, quantity \eqref{eq:funzionale} is conserved. Indeed
%\[
%\begin{split}
%\frac{d}{dt} \mathcal{H}[u,v]  &= \frac{d}{dt}  \int_{\Omega} \left( \frac{1}{2} v^2 + \frac{1}{2} \nabla u \cdot \K  \nabla u + f(u) %\right)  \, {\rm d}x = \int_{\Omega} \left( v_t \, v + \nabla u \cdot \K  \nabla u_t + f'(u)\, u_t \right)  \, {\rm d}x =\\
%&= \int_{\Omega} \bigl( (\nabla \cdot \K \nabla u - f'(u)) \, v + \nabla u \cdot \K  \nabla v + f'(u)\, v \bigr)  \, {\rm d}x =\\
%&= \int_{\Omega} \bigl( (\nabla \cdot \K \nabla u - f'(u)) \, v - \left(\nabla \cdot \K \nabla u \right) \, v + f'(u)\, v \bigr)  \, %{\rm d}x  + \int_{\partial \Omega} v \, \K\nabla u \cdot \mathbf{n} \, {\rm d}s = 0
%\end{split}
%\end{displaymath}
%obtaining the \textbf{global conservation law}
\begin{equation}
\label{eq:global laws}
\frac{d}{dt}\mathcal{H}[u,v] = 0 \ .
\end{equation}
The \textbf{energy density} of the system is defined by
\begin{equation}
\label{eq:energia}
E(u,v) := \frac{1}{2} v^2 + \frac{1}{2} \nabla u \cdot \K  \nabla u + f(u) \,.
\end{equation}
The total derivative of $E(u,v)$ with respect to $t$, along the solution $(u,v)$ of  \eqref{eq:ondesym}, is given by
\begin{displaymath}
E_t  = ({\rm div} \K \nabla u)v  + \nabla u \cdot \K \nabla v = {\rm div} \left( v \, \K \nabla u\right).
\end{displaymath}
Let $\boldsymbol{\omega}(u,v):= - v \, \K \nabla u$ the \textbf{energy flux}, then we have the \textbf{energy conservation law}
\begin{equation}
\label{eq:energy law}
E_t(u,v) + {\rm div} \, \boldsymbol{\omega}(u,v) = 0\ ,
\end{equation}
which is more general than the global conservation of the Hamiltonian. Indeed if the energy conservation law holds, then it is easy to prove that $\frac{d}{dt}\mathcal{H}[u,v]=0$.

% , then for the Divergence Theorem 
%\[
%\begin{split}
%\frac{d}{dt}\mathcal{H}[u,v] &= \frac{d}{dt} \int_{\Omega} E(u,v) \, {\rm d}x = \int_{\Omega} E_t(u,v) \, {\rm d}x = -  \int_{\Omega} %\nabla \cdot \mathbf{\omega}(u,v) \, {\rm d}x =\\
%&=  - \int_{\partial \Omega}  \mathbf{\omega}(u,v) \, {\rm d}s =  \int_{\partial \Omega}  v \, \K \nabla u \, {\rm d}s = 0.
%\end{split}
%\end{displaymath}

%\begin{osservazione}
%\label{oss:pre}
%In the literature usually  the mono-dimensional problem is treated with periodic conditions on the %solution instead of homogeneous boundary conditions. If we  set the problem \eqref{eq:ondesym} in a square %domain $\Omega$  and consider periodic boundary conditions (in both variables), then Hamiltonian is still %preserved, and the \textbf{momentum}:
%\[
%\mathcal{M}[u] := \int_{\Omega} v \nabla u {\rm d}x\ ,
%\end{displaymath}
%is a first integral of the problem, i.e.
%\[
%\frac{d}{dt}\mathcal{M}[u] = 0 \,.
%\end{displaymath}
%\end{osservazione}
%------------------------------------------------------------------------------------------------------------------------------

% -----------------------------------------------------------------------------------------------------------------------------

\section{The semi-discrete problem}
\label{sec:4}
By using the MFD approach we can  approximate the continuous operators by discrete ones, in  order to derive the semi-discrete problem for the wave  \eqref{eq:onde}. Then the resulting \textbf{semi-discrete wave equation} reads:
\begin{equation}
\label{eq:onde semi}
\left \{
\begin{aligned} 
& u_{h, tt}(t) = \Delta_h \, u_h(t) - f'(u_h(t)) & \text{for $t \in (0,T)$,}\\
&u_h(0) = u_{h,0}, \quad u_{h,t}(0) = v_{h,0}  \ ,
\end{aligned} 
\right .
\end{equation}
%& \text{find $u_h \in L^2(0,T, \mathcal{C}_h)$ and $u_{h,t} \in L^2(0,T, \mathcal{C}_h)$ such that} \\
where $u_{h,0} := u_0^I$ and $v_{h,0} := v_0^I$ are the interpolant functions in $\mathcal{C}_h$ of the initial data. In the same way,  \eqref{eq:ondesym} can be discretized in the following form
\begin{equation}
\label{eq:ondesym semi}
\left \{
\begin{aligned} 
& u_{h, t}(t) = v_h(t)  & \text{for  $t \in (0,T)$,}\\
& v_{h, t}(t) = \Delta_h\, u_h(t) - f'(u_h(t)) & \text{for  $t \in (0,T)$,}\\
&u_h(0) = u_{h,0}, \quad v_h(0) = v_{h,0}  \ .
\end{aligned} 
\right .
\end{equation}
%& \text{find $u_h \in L^2(0,T, \mathcal{C}_h)$ and $v_h \in L^2(0,T, \mathcal{C}_h)$ such that} \\
We observe that  the semi-discrete  \eqref{eq:ondesym semi} preserves the Hamiltonian structure of  \eqref{eq:ondesym}. In light of the definition in  Section \ref{sec:2}, the Hamiltonian functional $\mathcal{H}$ in  \eqref{eq:funzionale} admits the natural {\bf mimetic} semi-discretization:
\begin{equation}
\label{eq:funzionalesemi}
\mathcal{H}_h[u_h, v_h] := \frac{1}{2} [v_h, v_h]_{\mathcal{C}_h} + \frac{1}{2} [{\gd} \,u_h, {\gd}\, u_h]_{\mathcal{F}_h} + [f(u_h), 1]_{\mathcal{C}_h}\ ,
\end{equation}
that will be called \textbf{mimetic semi-discrete Hamiltonian  functional}.

We can observe now that, if we denote with $\nabla_{v_h}$ the gradient with respect to the variable $v_h$ and with $\nabla_{u_h}$ the gradient with respect to the variable $u_h$, then
\begin{displaymath}
\M_{\mathcal{C}_h}^{-1} \, \nabla_{v_h} \mathcal{H}_h[u_h, v_h] = \M_{\mathcal{C}_h}^{-1} \M_{\mathcal{C}_h} v_h = v_h \ ,
\end{displaymath}
and
\begin{equation*}
\begin{split}
\M_{\mathcal{C}_h}^{-1} \, \nabla_{u_h} \mathcal{H}_h[u_h, v_h]  &= \M_{\mathcal{C}_h}^{-1} \left( {\gd}^T \, \M_{\mathcal{F}_h}\, {\gd} \, u_h  + \M_{\mathcal{C}_h} f'(u_h) \right) = -\Delta_h \,u_h + f'(u_h).
\end{split}
\end{equation*}
Hence  \eqref{eq:onde semi} may be written as a Hamiltonian system of ordinary differential equations (ODEs), that is as:
\begin{displaymath}
\begin{pmatrix}
u_{h,t} \\ v_{h,t}
\end{pmatrix}
=
\mathcal{J}_{N_c} \M_{\mathcal{C}_h}^{-1}  \nabla  \mathcal{H} [u_h, v_h]\ ,
\end{displaymath}
where $\mathcal{J}_{N_c}$ is the canonical symplectic matrix  $\bigl( \begin{smallmatrix}& 0 & I_{N_c}\\ & -I _{N_c}& 0  \end{smallmatrix}\bigr)$ while $\nabla$ denotes the gradient with respect the variables $u_h$ and $v_h$.

We can conclude that the MFD approach gives a finite-dimensional system of ODEs that retains the Hamiltonian character of the given PDE. Therefore MFD methods can be considered powerful scheme for the spatial discretization of Hamiltonian PDEs.

% --------------------------------------------------------------------------------------------------------------------------------------------------
% --------------------------------------------------------------------------------------------------------------------------------------------------

\subsection{Convergence for the semi-discrete problem}
\label{sub:4.2}

Now we will investigate the convergence of the solution $u_h$ of the semi-discrete wave equation  \eqref{eq:onde semi} to the solution $u$ of   \eqref{eq:ondesym} in $L^2(\Omega)$ norm. Before analysing the error between the solution, we have to show some preliminary technical results.

Let us introduce the \textbf{energy projection} $\mathcal{P}_h \colon H^2(\Omega) \to \mathcal{C}_h$, with $u \mapsto \mathcal{P}_h \, u$ defined as the solution of the diffusion problem
\begin{equation}
\label{eq:proiezione}
\left \{
\begin{aligned} 
& \text{find $\mathcal{P}_h\,u \in  \mathcal{C}_h$ and $\boldsymbol{\sigma}_h \in \mathcal{F}_h$ such that} \\
& [\boldsymbol{\sigma}_h , \boldsymbol{\omega}_h]_{\mathcal{F}_h} + [\mathcal{P}_h\, u, {\did} \,\boldsymbol{\omega}_h]_{\mathcal{C}_h} = 0 \qquad & \text{for all $\boldsymbol{\omega}_h \in \mathcal{F}_h$,} \\
& [{\did}\, \boldsymbol{\sigma}_h , w_h]_{\mathcal{C}_h} = [({\rm div} \K \nabla u)^I, w_h]_{\mathcal{C}_h}  \qquad & \text{for all $w_h \in \mathcal{C}_h$.} \\ 
\end{aligned} 
\right .
\end{equation}
In particular, for the duality relation between the operators, the projection $\mathcal{P}_h \, u$ satisfies
\begin{equation}
\label{eq:lapint}
\Delta_h (\mathcal{P}_h \, u) = ({\rm div} \K \nabla u)^I.
\end{equation}

In the following  we use $\|\cdot \|_{\mathcal{C}_h}$ to denote the norm induced by the scalar product $[\cdot, \cdot]_{\mathcal{C}_h}$ (that is equivalent to $L^2(\Omega)$ norm on $\mathcal{C}_h$). We will denote with  $C$ a generic constant, possibly different at each occurrence, independent from the mesh size $h$ and the time step size $\tau$. In order to prove the convergence results, we need the following Lemma (see \cite{MR2192322} for the proof).
\begin{lemma}
\label{lemma1}
Let us assume the convexity of the domain $\Omega$.  Let $u \in H^2(\Omega)$ and let $\mathcal{P}_h\,u$ the energy projection of $u$. Then the following estimate holds:
\begin{displaymath}
\|u^I - \mathcal{P}_h\,u \|_{\mathcal{C}_h} \leq C h^2|u|_{H^2(\Omega)}.
\end{displaymath}
\end{lemma}
While the next Lemma shows the spectral properties of the operator $\Delta_h$(see \cite{vacca} for more details).
\begin{lemma}
\label{lemma2}
The spectrum $\sigma(-\Delta_h)$ of $-\Delta_h$ satisfies
\begin{equation}
\label{eq:estremali}
\sigma(-\Delta_h) \subseteq [s_*,\, s^*h^{-2}],
\end{equation}
where $s_*$ and $s^*$ are positive and $h$-independent constants.
\end{lemma}

For the treatment of the nonlinear term we have the following lemma.

\begin{lemma}
\label{lemma4}
Let $f$ be a smooth function and let $u \in H^2(\Omega)$. Then
\begin{displaymath}
\|f'(u)^I - f'(u^I)\|_{\mathcal{C}_h} \leq  C(u) \, h^2.  
\end{displaymath}
\end{lemma}
\begin{proof}
Let $u_c := (u^I)_c$ for every $c \in \mathcal{T}_h$. Then the  nonlinear term may be treated in the following way. Using the Taylor expansion, since $u \in H^2(\Omega)$ 
\begin{equation}
\label{eq:ftaylor}
f'(u(x)) = f'(u_c) + f''(u_c)(u(x) - u_c) + \frac{1}{2} f'''(\widehat{u}_c(x))(u(x) - u_c)^2 \qquad \text{for a.e. $x \in c$}
\end{equation}
for suitable $\widehat{u}_c(x)$ and  for every $c \in \mathcal{T}_h$. Then, setting $f(u)_c := \left( f(u)^I\right)_c$ and using  \eqref{eq:ftaylor}, we have:
\begin{equation*}
\begin{split}
f'(u)_c  - f'(u_{c})  &= \frac{1}{|c|} \int_{c} \left (f'(u(x)) - f'(u_{c}) \right) \,{\rm d}x \\
&= \frac{1}{|c|} \int_{c} \left( f'(u_c) + f''(u_c)(u(x) - u_c) + \frac{1}{2} f'''(\widehat{u}_c(x))(u(x) - u_c)^2 - f'(u_{c})\right) \, {\rm d}x. 
\end{split}
\end{equation*}
Now, since $f''(u_c)$ is constant and, by definition, $u_c = \frac{1}{|c|}\int_c u(x)\, {\rm d} x$, we obtain
\begin{displaymath}
f'(u)_c  - f'(u_{c})  = \sigma(u)_c
\end{displaymath}
with
\begin{equation}
\label{eq:sigma}
\sigma(u)_c :=  \frac{1}{2 \, |c|} \int_{c}  f'''(\widehat{u}_c(x))(u(x) - u_c)^2 \, {\rm d} x. 
\end{equation}
Now we observe that $u \in H^2(\Omega)$ for classic Sobolev embedding theory, implies that $u \in L^{\infty}(\Omega)$, then, since $\widehat{u}_c(x)$ is bounded by $u(x)$ and the constant value $u_c$ we obtain that $\widehat{u}_c \in L^{\infty}(c)$ for all element $c \in \mathcal{T}_h$. Therefore being $f$ a smooth function, $f'''(\widehat{u}_c) \in L^{\infty}(c)$.
Now, using the H{\"o}lder Theorem 
\begin{equation}
\label{eq:postsigma1}
\begin{split}
\|\sigma(u,t) \|_{\mathcal{C}_h}  &= \frac{1}{2} \sum_{c \in \mathcal{T}_h}  \left|\int_{c}  f'''(\widehat{u}_c(x))(u(x) - u_c)^2 \, {\rm d}x \right| \leq 
\frac{1}{2} \sum_{c \in \mathcal{T}_h} \int_{c}  |f'''(\widehat{u}_c(x))| (u(x) - u_c)^2 \, {\rm d}x \\
& \leq \frac{1}{2} \sum_{c \in \mathcal{T}_h}  \|f'''(\widehat{u}_c)\|_{L^{\infty}(c)} \| (u - u_c)^2\|_{L^1(c)} \leq C \, \sum_{c \in \mathcal{T}_h} \| (u - u_c)^2\|_{L^1(c)}. 
\end{split} 
\end{equation}
Now, using standard polynomial approximation results \cite{scott}, we have
\begin{equation}
\label{eq:postsigma}
\begin{split}
\|\sigma(u,t) \|_{\mathcal{C}_h}  &\leq C \sum_{c \in \mathcal{T}_h} \| (u - u_c)^2\|_{L^1(c)} = C \sum_{c \in \mathcal{T}_h} \| (u - u_c)\|^2_{L^2(c)} \\
&\leq C \sum_{c \in \mathcal{T}_h} h^2 \, |u|^2_{H^1(c)} = C \, h^2 \, |u|^2_{H^1(\Omega)}. 
\end{split}
\end{equation}
\end{proof}

Now we have the instruments for proving the following convergence theorem.
\begin{teorema}
\label{thm:convergenza}
Under the assumptions of  Lemma \ref{lemma1} and  Lemma \ref{lemma2}, let $u(x,t)$ be the solution of  \eqref{eq:onde} and $u_h(t)$ be the solution of  \eqref{eq:onde semi}. 
Let us assume that $u(\cdot, t) \in H^2(\Omega)$ for all $t \in [0, T]$.
Moreover  let us assume that $f'$ is globally Lipschitz. 
Then, for all $t \in [0,T]$, it follows that:
\begin{multline*}
\|u(t)^I - u_h(t) \|_{\mathcal{C}_h} \leq C\,  \psi(T) \, h^2 \, \bigl( |u_0|_{H^2(\Omega)} + |v_0|_{H^2(\Omega)} + |u_t|_{L^1(0,t, H^2(\Omega))}+ \bigr. \\
+ \left. |u_{tt}(t)|_{L^2(0,t, H^2(\Omega))} + |u(t)|_{L^2(0,t,H^2(\Omega))} + |u(t)|^2_{L^2(0,t,H^1(\Omega))}\right),
\end{multline*}
where $u(t)^I$ denotes the interpolant of $u(x,t)$ in $\mathcal{C}_h$ and the scalar function $\psi(t)$ is bounded for all $t \in [0,T]$. 
%\label{eq:stima convergenza}
\end{teorema}

\begin{proof}
The proof follows the guidelines of Theorem~1 in \cite{nonlinear} for given for the finite element approximation.
Let us set
\begin{equation}
\label{eq:first}
u_h(t) - u(t)^I = \left( u_h(t) - \mathcal{P}_h \, u(t) \right) + \left( \mathcal{P}_h \, u(t) - u(t)^I \right) =:  \theta(t) + \rho(t).
\end{equation}
We study separately the two terms. The second term represents the error generated by the energy projection;  using  Lemma \ref{lemma1}, we obtain
\begin{equation}
\label{eq:rho}
\begin{split}
\| \rho(t)\|_{\mathcal{C}_h} &= \left \|\mathcal{P}_h \, u(t) - u(t)^I \right \|_{\mathcal{C}_h} \leq C h^2 |u(t)|_{H^2(\Omega)}
= C h^2 \left( |u(0)|_{H^2(\Omega)} + \int_0^t |u_t(s)|_{H^2(\Omega)} {\rm d}\, s \right)  \\
& \leq C h^2 \left( |u_0|_{H^2(\Omega)} + |u_t|_{L^1(0,t, H^2(\Omega))} \right).
\end{split}
\end{equation}
For the first term, from  \eqref{eq:onde semi} and  \eqref{eq:lapint}, we get
\begin{displaymath}
\theta_{tt}(t) - \Delta_h\, \theta(t) = -f'(u_h(t)) - (\mathcal{P}_h \, u(t))_{tt} + ({\rm div} \K \nabla u(t))^I 
\end{displaymath}
and, since $u$ is the solution of  \eqref{eq:onde}, we obtain
\begin{equation*}
\begin{split}
\theta_{tt}(t) -\Delta_h\, \theta(t) & = -f'(u_h(t)) - \mathcal{P}_h \, u_{tt}(t) + u_{tt}(t)^I + (f'(u(t)))^I \\
&= - \rho_{tt}(t) - \left( f'(u_h(t))- f'(u(t))^I \right)
\end{split}
\end{equation*}
and in particular 
\begin{equation}
\label{eq:theta}
\begin{split}
[\theta_{tt}(t), \,  \chi]_{\mathcal{C}_h}  - [\Delta_h\, \theta(t) , \, \chi]_{\mathcal{C}_h}  = 
-[ \rho_{tt}(t), \, \chi]_{\mathcal{C}_h}  - [ f'(u_h(t))- f'(u(t))^I, \, \chi]_{\mathcal{C}_h} 
\end{split}
\end{equation}
for all $\chi \in \mathcal{C}_h$. For $t \in [0, T]$ let use define
\begin{equation}
\label{eq:G}
G(t) := \int_0^t  - (f'(u_h(s))- f'(u(s))^I) \, {\rm d}s,
\end{equation}
and let $\chi = \chi(t) \in \mathcal{C}_h$ in  \eqref{eq:theta} be a function of $t$. Then it is straightforward to see that
\begin{multline}
\label{eq:theta1}
-[\theta_{t}(t), \,  \chi_t(t)]_{\mathcal{C}_h} - [\Delta_h\, \theta(t) , \, \chi(t)]_{\mathcal{C}_h}
 = \frac{d}{dt} [(u_{t}^I - u_{h,t})(t) + G(t), \,  \chi(t)]_{\mathcal{C}_h} +\\     
+[ \rho_{t}(t), \, \chi_t(t)]_{\mathcal{C}_h} - [ G(t), \, \chi_t(t)]_{\mathcal{C}_h}. 
\end{multline}
Let us fix $\tau \in [0,T)$ and we set in  \eqref{eq:theta1} 
\begin{displaymath}
\chi(t) := \int_t^{\tau} \theta(s) \, {\rm d}s, \qquad \text{for $t \in [0,T]$,}
\end{displaymath}
in particular we can observe that $\chi_t(t) = - \theta(t)$. Now the duality relation among discrete operators and simple computations yield
\begin{multline*}
[\theta_{t}(t), \,  \theta(t)]_{\mathcal{C}_h} - [\gd \, \chi_t(t) , \, \gd \, \chi(t)]_{\mathcal{F}_h}
 = \frac{d}{dt} [(u_{t}^I - u_{h,t})(t) + G(t), \,  \chi(t)]_{\mathcal{C}_h} +\\     
-[ \rho_{t}(t), \, \theta(t)]_{\mathcal{C}_h} + [ G(t), \, \theta(t)]_{\mathcal{C}_h}
\end{multline*}
thus
\begin{equation}
\label{eq:theta2}
\frac{1}{2} \frac{d}{dt}\|\theta(t)\|^2_{\mathcal{C}_h} -\frac{1}{2} \frac{d}{dt} \|\gd \, \chi(t)\|^2_{\mathcal{F}_h}
 = \frac{d}{dt} [(u_{t}^I - u_{h,t})(t) + G(t), \,  \chi(t)]_{\mathcal{C}_h}      
+ [G(t) - \rho_{t}(t), \, \theta(t)]_{\mathcal{C}_h}.
\end{equation}
Integrating  \eqref{eq:theta2} with respect to $t$ from 0 to $\tau$, observing that $\chi(\tau) = 0$, $G(0)=0$ and by definition $u_{h,t}(0) = u_t(0)^I$, we get
\begin{displaymath}
\|\theta(\tau)\|^2_{\mathcal{C}_h} - \|\theta(0)\|^2_{\mathcal{C}_h} + \|\gd \, \chi(0)\|^2_{\mathcal{F}_h}
 =      
 2 \int_0^{\tau} [G(t) - \rho_{t}(t), \, \theta(t)]_{\mathcal{C}_h} \, {\rm d}t
\end{displaymath}
and then
\begin{equation}
\label{eq:theta3}
\|\theta(\tau)\|^2_{\mathcal{C}_h} \leq  \|\theta(0)\|^2_{\mathcal{C}_h} 
+ 2 \int_0^{\tau} \| G(t)\|_{\mathcal{C}_h} \|\theta(t)\|_{\mathcal{C}_h}  {\rm d}t 
+ 2 \int_0^{\tau} \| \rho(t)\|_{\mathcal{C}_h} \|\theta(t)\|_{\mathcal{C}_h}  {\rm d}t.  
\end{equation}
Now by definition  \eqref{eq:G}, from Lipschitz assumption on the load $f'$,  and  Lemma \ref{lemma4} we have
\begin{equation*}
\begin{split}
\| G(t)\|_{\mathcal{C}_h} & \leq  \int_0^t \|f'(u_h(s))  - f'(u(s))^I \|_{\mathcal{C}_h} \,{\rm d}s \\
& \leq  \int_0^t \|f'(u_h(s))  - f'(u(s)^I) \|_{\mathcal{C}_h} \,{\rm d}s +   \int_0^t \|f'(u(s))^I  - f'(u(s)^I) \|_{\mathcal{C}_h} \,{\rm d}s  \\
& \leq  C \int_0^t \| \theta(s) \|_{\mathcal{C}_h} \,{\rm d}s +   C \int_0^t \|\rho(s) \|_{\mathcal{C}_h} \,{\rm d}s + C \, h^2 \, \int_0^t |u(s) |^2_{H^1(\Omega)} \,{\rm d}s. 
\end{split}
\end{equation*}
Therefore, from Cauchy-Swartz inequality and  \eqref{eq:rho}
\begin{equation*}
\begin{split}
2 & \int_0^{\tau} \| G(t)\|_{\mathcal{C}_h} \|\theta(t)\|_{\mathcal{C}_h}  {\rm d}t  \\
& \leq 2 C  \int_0^{\tau} \left(\int_0^t \| \theta(s) \|_{\mathcal{C}_h} \,{\rm d}s +    \int_0^t \|\rho(s) \|_{\mathcal{C}_h} \,{\rm d}s +   h^2 \, \int_0^t |u(s) |^2_{H^1(\Omega)} \,{\rm d}s \right) \|\theta(t)\|_{\mathcal{C}_h} \,{\rm d} t \\
&  \leq  C\int_0^{\tau} \left( \int_0^t \|\rho(s) \|_{\mathcal{C}_h} \,{\rm d}s +   h^2 \, \int_0^t |u(s) |^2_{H^1(\Omega)} \,{\rm d}s \right)^2  \,{\rm d} t +  C\int_0^{\tau} (1 + \tau)  \|\theta(t)\|^2_{\mathcal{C}_h} \,{\rm d} t \\
&  \leq  C(u) \, T \,h^4 \, + C(1 + T) \, \int_0^{\tau} \|\theta(t)\|^2_{\mathcal{C}_h} \,{\rm d} t 
\end{split}
\end{equation*}
and always from Cauchy-Swartz and  \eqref{eq:rho}
\begin{displaymath}
2 \int_0^{\tau} \| \rho(t)\|_{\mathcal{C}_h} \|\theta(t)\|_{\mathcal{C}_h}  {\rm d}t \leq
\int_0^{\tau} \| \rho(t)\|^2_{\mathcal{C}_h}  {\rm d}t + \int_0^{\tau} \|\theta(t)\|^2_{\mathcal{C}_h}{\rm d}t \leq C(u) \, T \, h^4 + \int_0^{\tau} \|\theta(t)\|^2_{\mathcal{C}_h} \,{\rm d} t.
\end{displaymath}
By collecting the previous estimates in  \eqref{eq:theta3}, from  \eqref{eq:rho} we get
\begin{equation}
\label{eq:theta4}
\begin{split}
\|\theta(\tau)\|^2_{\mathcal{C}_h} & \leq  \|\theta(0)\|^2_{\mathcal{C}_h} + C(u) \, T \, h^4 + (1 + C + C \, T)  \int_0^{\tau} \|\theta(t)\|^2_{\mathcal{C}_h} \,{\rm d} t.
\end{split}
\end{equation}
It is straightforward to check that
\begin{displaymath}
 \|\theta(0)\|_{\mathcal{C}_h} =  \|\rho(0)\|_{\mathcal{C}_h} \leq C(u_0) \, h^2,
\end{displaymath}
then by Gronwall inequality it holds that
\begin{displaymath}
\|\theta(\tau)\|^2_{\mathcal{C}_h}  \leq C(u, u_0) \, h^4 \,T e^{LT}.
\end{displaymath}
from which follows the thesis.

\end{proof}

\begin{osservazione} The use of the projection $\mathcal{P}_h \, u$
 in the proof of the theorem seems to be necessary. Indeed if we compute directly $u_h(t) - u(t)^I$ as done for example in  \cite{MR2586202}, we  obtain a term of the form
\begin{displaymath}
L(u):= \left  \|({\rm div} \K \nabla u)^I - \Delta_h \, u^I \right \|_{\mathcal{C}_h} \ ,
\end{displaymath}
and $L(u)$ does not converge to zero. For instance in  Figure \ref{fig:1} we plot the asymptotic behaviour of $L(u)$ as a function of $h$ for $u(x,y)= \sin(\pi x)\sin(\pi y)$, tensor $\K = I_2$ and domain $\Omega = [0,1] \times [0,1]$ discretized with the sequence of Voronoi meshes introduced in Section \ref{sec:6}, see  Figure \ref{fig:2}. The value of $L(u)$ does not seem to converge to zero as $h$ is reduced.

\begin{figure}[!h]
\centering
\includegraphics[scale=0.25]{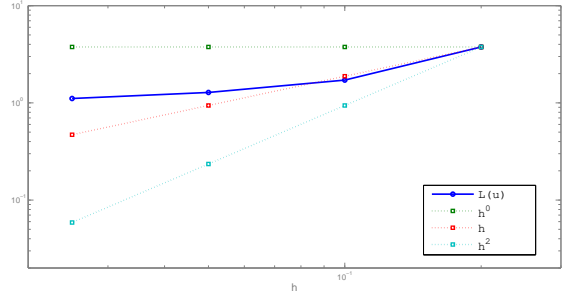}
\caption{Asymptotic behaviour of $L(u)$ as a function of $h$.}
\label{fig:1}
\end{figure}

\end{osservazione}

\begin{osservazione}
\label{oss:local_lip}
Using standard theory of polynomial approximations and the definition of the projection $\mathcal{P}_h$ (see \cite{nonlinear}), it is possible to extend the proof of Theorem \ref{thm:convergenza} to the case of $f'$ local Lipschitz instead of global Lipschitz. The proof being analogous but  more technical.
\end{osservazione}

\subsection{Conservation laws for the semi-discrete problem}
As for the continuous system, it is easy to prove  that the \textbf{global semi-discrete conservation law} of the Hamiltonian semi-discrete functional $\mathcal{H}_h[u_h, v_h]$ is preserved. Indeed using the duality definitions of the discrete operators, we have
\begin{equation}
\label{eq:semi global low}
\begin{split}
\frac{d}{dt}& \mathcal{H}_h[u_h, v_h]  = \frac{d}{dt} \left( \frac{1}{2} [v_h, v_h]_{\mathcal{C}_h} + \frac{1}{2} [{\gd} \,u_h, {\gd} \,u_h]_{\mathcal{F}_h} + [f(u_h), 1]_{\mathcal{C}_h} \right) \\
& = [v_{h,t}, v_h]_{\mathcal{C}_h} + [{\gd} \,u_{h,t}, {\gd}\, u_h]_{\mathcal{F}_h} + [f'(u_h)\, u_{h,t}, 1]_{\mathcal{C}_h}  \\
& = [\Delta_h \,u_h, v_h]_{\mathcal{C}_h} - [f'(u_h), v_h]_{\mathcal{C}_h}+ [{\gd} \,v_h, {\gd} \,u_h]_{\mathcal{F}_h} + [f'(u_h), v_h]_{\mathcal{C}_h}  \\
& = -[{\gd} \,u_h, {\gd} \,v_h]_{\mathcal{F}_h} + [{\gd} \,v_h, {\gd} \,u_h]_{\mathcal{F}_h} = 0\ ,
\end{split}
\end{equation}
along the solution $(u_h(t) , v_h(t))$ of  \eqref{eq:onde semi}.

%Therefore we have the \textbf{global semi-discrete conservation law}
%\begin{equation}
%\label{eq:semi global low}
%\frac{d}{dt} \mathcal{H}_h[u_h, v_h] = 0.
%\end{equation}

We can define the \textbf{mimetic semi-discrete energy density} $E_h \in \mathcal{C}_h$ with
\begin{equation}
\label{eq:semi energia}
E_h(u_h, v_h)_{|_c} := \frac{1}{2} |c| v_{h,c}^2 + \frac{1}{2} \left [({\gd} \,u_h)_c , ({\gd} \,u_h)_c \right]_{\mathcal{F}_{h,c}} + |c|\,f(u_{h,c})\ ,
\end{equation}
and by computing its derivative with respect to $t$ along the solution, we have
\begin{equation*}
\begin{split}
\frac{d}{dt} & E_h(u_h, v_h)_{|_c} = \frac{d}{dt} \left( \frac{1}{2}|c|\, v_{h,c}^2 + \frac{1}{2} \left [({\gd} \,u_h)_c , ({\gd} \,u_h)_c \right]_{\mathcal{F}_{h,c}} + |c|\, f(u_{h,c})\right)\\
& = |c| \,v_{h,c} \, \frac{d}{dt} v_{h,c}  + \left[\frac{d}{dt}({\gd} \,u_h)_c, ({\gd} \,u_h)_c \right]_{\mathcal{F}_{h,c}} + |c|\, f'(u_{h,c})\frac{d}{dt} u_{h,c} \\
& = |c|(\Delta_h \,u_{h})_c \, v_{h,c} -|c|\, f'(u_{h,c}) \, v_{h,c} + \left[({\gd}\, v_h)_c, ({\gd}\, u_h)_c \right]_{\mathcal{F}_{h,c}} + |c|\,f'(u_{h,c})v_{h,c} \\
& = |c|(\Delta_h\, u_{h})_c \, v_{h,c} +  \left[({\gd} \,v_h)_c, ({\gd}\, u_h)_c \right]_{\mathcal{F}_{h,c}}.
\end{split}
\end{equation*}
%Let $(\mathbf{\omega}_h, q_h)$ be the solution of the diffusion problem 
% \[
% \begin{sistema}
% \mathbf{\omega}_h + {\gd} q_h = \psi_h \\
% - {\did} \mathbf{\omega}_h = E_h
% \end{sistema}
% \end{displaymath}
Then, the following \textbf{mimetic semi-discrete energy conservation law} holds:
\begin{equation}
\label{eq:semi energy law}
E_{h,t}(u_h , v_h) + F_h(u_h, v_h) = 0
\end{equation}
where $F_h(u_h, v_h) \in \mathcal{C}_h$, defined by
\begin{displaymath}
F_{h,c}(u_h, v_h) = - |c|(\Delta_h \,u_{h})_c \, v_{h,c} -  \left[({\gd}\, v_h)_c, ({\gd}\, u_h)_c \right]_{\mathcal{F}_{h,c}}\ ,
\end{displaymath}
is a natural discretization of ${\did}\, (v_h \, {\gd} \, u_h)$.  \eqref{eq:semi energy law} represents the mimetic approximation of the energy conservation law  \eqref{eq:energy law}.

We have observed that the continuous Hamiltonian and semi-discrete Hamiltonian are first integrals respectively for system  \eqref{eq:ondesym} and  \eqref{eq:ondesym semi}. In particular, for all $t \in [0,T]$ we have
\begin{displaymath}
\mathcal{H}[u(t), v(t)] = \mathcal{H}[u_0, v_0] =: \mathcal{H}_0, \qquad \text{and} \qquad  \mathcal{H}_h[u_h(t), v_h(t)] = \mathcal{H}[u_{h,0}, v_{h,0}] =: \mathcal{H}_{h,0}
\end{displaymath}
where $(u(t), v(t))$ is the solution of  \eqref{eq:ondesym} and $(u_h(t), v_h(t))$ is the solution of  \eqref{eq:ondesym semi}.
In the following results we estimate the error between  the continuous Hamiltonian and the semi-discretized Hamiltonian.

\begin{lemma}
\label{lemma3}
Let $u \in H_0^3(\Omega)$,  and $v \in H_0^1(\Omega)$ and let $u^I$ and  $v^I \in \mathcal{C}_h$  their respective  interpolant function in $\mathcal{C}_h$. Then it follows that:
\begin{equation}
\label{eq:haminter}
\bigl|  \mathcal{H}[u,v] - \mathcal{H}_h[u^I,v^I] \bigr| \leq C \, h^2 \left(|u|^2_{H^1(\Omega)} + |v|^2_{H^1(\Omega)} + |u|_{H^1(\Omega)}|u|_{H^3(\Omega)} \right) 
\end{equation}
\end{lemma}

\begin{proof}
We split the bound for the three terms composing the Hamiltonian functional. 
Let us start with
\begin{equation}
\label{eq:ham3}
\left |\int_{\Omega} f(u(x)) \, {\rm d}x - [f(u^I), 1]_{\mathcal{C}_h}  \right| \ ,
\end{equation}
and we observe that since $u \in L^2(\Omega)$, using the same computations in  \eqref{eq:ftaylor} cell by cell
\begin{displaymath}
f(u(x)) = f(u_c) -  f'(u_c) (u(x) - u_c) + \frac{1}{2} f''(\widetilde{u}_c(x))(u(x) - u_c)^2
\end{displaymath}
where we observe that using the same arguments in the proof of  Lemma \ref{lemma4} the term $f''(\widetilde{u}_c) \in L^{\infty}(c)$. 
Now, since by definition $u_c = \frac{1}{|c|} \int_c u(x) \, {\rm d}x$, using the same computation in  \eqref{eq:postsigma1} and  \eqref{eq:postsigma}, it follows that
\begin{multline*}
\left |\int_{c} f(u(x)) \, {\rm d}x - [f(u_c), 1]_{\mathcal{C}_{h,c}} \right| = \left| \int_c (f(u(x)) - f(u_c)) \, {\rm d}x  \right|  =\\
 = \left| \int_c \frac{1}{2} f''(\widetilde{u}_c(x))(u(x) - u_c)^2 \, {\rm d}x  \right|  \leq C \, h^2 \, |u|^2_{H^1(c)} . 
\end{multline*}

By adding in the cell $c \in \mathcal{T}_h$ we bound the first term  \eqref{eq:ham3} as follows
\begin{displaymath}
\left|\int_{\Omega} f(u(x)) \, {\rm d}x - [f(u^I), 1]_{\mathcal{C}_h}  \right| \leq C h^2 \sum_{c \in \mathcal{T}_h}|u|^2_{H^1(c)} \leq C \, h^2 \, |u|^2_{H^1(\Omega)}\ .
\end{displaymath}
For the term
\begin{equation}
\label{eq:ham2}
\left |\int_{\Omega} v^2 \, {\rm d}x - [v^I , v^I]_{\mathcal{C}_h}  \right| \ ,
\end{equation}
we observe that $[v^I, v^I]_{\mathcal{C}_h} = \left [(v^I)^2, 1\right]_{\mathcal{C}_h} $, and thus we can use the computations used before with $f(s) = s^2$.

For the last term, we preliminary observe that, using integration by parts and homogeneous boundary conditions, and since ${\gd} = -{\did}^*$, we have to estimate
\begin{displaymath}
\int_{\Omega}  ({\rm div} \K \nabla u)\, u \,{\rm d}x - [\Delta_h \, u^I, u^I]_{\mathcal{C}_h}.
\end{displaymath}
Now, by definition of interpolation operator in $\mathcal{C}_h$ and  \eqref{eq:lapint}, we get
\begin{equation}
\label{eq:ham1}
\begin{split}
&\int_{\Omega} ({\rm div} \K \nabla u) \, u \,{\rm d}x - [\Delta_h\, u^I, u^I]_{\mathcal{C}_h}=\\
&= \int_{\Omega} \left({\rm div} \K \nabla u - \left( {\rm div} \K \nabla u \right)^I  \right) \,  (u - u^I) \,{\rm d}x + \int_{\Omega}  ({\rm div} \K \nabla u ) \, u^I  \,{\rm d}x -  [\Delta_h\, u^I, u^I]_{\mathcal{C}_h} \\
&=: \alpha  +  \left [\left( {\rm div} \K \nabla u \right)^I -  \Delta_h \, u^I, u^I \right]_{\mathcal{C}_h}  \\
&= \alpha  +  \left [ \Delta_h \, (\mathcal{P}_h \,u - u^I), u^I -  \mathcal{P}_h \,u\right]_{\mathcal{C}_h} + \left [ \Delta_h \, (\mathcal{P}_h \,u - u^I),  \mathcal{P}_h \,u\right]_{\mathcal{C}_h}\\
&= \alpha  +  \left [ \Delta_h \, (\mathcal{P}_h \,u - u^I), u^I -  \mathcal{P}_h \,u\right]_{\mathcal{C}_h} + \left [  \mathcal{P}_h \,u - u^I,  \left( {\rm div} \K \nabla u \right)^I \right]_{\mathcal{C}_h} \\
&=: \alpha + \beta + \gamma.
\end{split}
\end{equation}
Using standard polynomial approximation estimates \cite{scott}, we have
\begin{equation}
\label{eq:alphaham}
|\alpha| = \left|  \int_{\Omega}  \left({\rm div} \K \nabla u - \left( {\rm div} \K \nabla u \right)^I  \right) \, (u - u^I) \,{\rm d}x \right| \leq C \, h^2 \,|u|_{H^1(\Omega)} \, |u|_{H^3(\Omega)}. 
\end{equation}
Moreover, by  Lemma \ref{lemma1} and  Lemma \ref{lemma2}, we get
\begin{equation}
\label{eq:betaham}
\begin{split}
|\beta| &= \left| \left [\Delta_h \, (\mathcal{P}_h \,u - u^I), u^I -  \mathcal{P}_h \,u\right]_{\mathcal{C}_h}\right|  \leq \|\Delta_h\| \|u^I -  \mathcal{P}_h \,u\|_{\mathcal{C}_h}^2 \\
&  \leq C \, h^{-2} \, h^4 \, |u|^2_{H^2(\Omega)}   = C\, h^2 \, |u|^2_{H^2(\Omega)}.
\end{split}
\end{equation} 
For the last term $\gamma$, we it holds
\begin{equation}
\label{eq:gammaham}
|\gamma| = \left| \left [  \mathcal{P}_h \,u - u^I,  \left( {\rm div} \K \nabla u \right)^I \right]_{\mathcal{C}_h}\right| \leq C  \, |u|_{H^2(\Omega)} \, \| \mathcal{P}_h \,u - u^I\|_{\mathcal{C}_h} \leq C\, h^2 \, |u|^2_{H^2(\Omega)}.
\end{equation}
Finally, by collecting  \eqref{eq:alphaham},  \eqref{eq:betaham} and  \eqref{eq:gammaham} in  \eqref{eq:ham1}, we obtain 
\begin{equation}
\label{eq:ham1def}
\left| \int_{\Omega}  ({\rm div} \K \nabla u) \, u  \,{\rm d}x - [\Delta_h \, u_h, u_h]_{\mathcal{C}_h} \right| \leq C \, h^2 \left(|u|^2_{H^2(\Omega)} + |u|_{H^1(\Omega)} \, |u|_{H^3(\Omega)} \right).
\end{equation}
Finally, the thesis follows from   \eqref{eq:ham3},  \eqref{eq:ham2} and  \eqref{eq:ham1def}.
\end{proof}

Since in order to define the initial data in the MFD discretization we used the interpolantion operator in $\mathcal{C}_h$,   Lemma \ref{lemma3} implies the following estimates on the error between the continuous and semi-discrete Hamiltonian.

\begin{teorema}
\label{thm:conservazione1}
Let $(u(t), v(t))$ be the solution of system  \eqref{eq:ondesym} and let $(u_h(t), v_h(t))$ be the solution of  \eqref{eq:ondesym semi}. Then, for all $t \in [0,T]$ it holds
\begin{equation}
\label{eq:hamerrorsemi}
\bigl | \mathcal{H}[u(t), v(t)]  - \mathcal{H}_h[u_h(t), v_h(t)] \bigr| \leq C h^2
\end{equation}
where the constant $C$ depends only on the regularity of the initial data $u_0$ and $v_0$.
\end{teorema}

\begin{proof}
For all $t \in [0;T]$, it holds that
\begin{displaymath}
\mathcal{H}[u(t), v(t)] = \mathcal{H}_0, \qquad \text{and} \qquad \mathcal{H}_h[u_h(t), v_h(t)] = \mathcal{H}_{h,0}.
\end{displaymath}
Now, since $u_{h,0} = u_0^I$ and $v_{h,0} = v_0^I$, from  Lemma \ref{lemma3} we get  the thesis.
\end{proof}

% -----------------------------------------------------------------------------------------------------------------------------
% -----------------------------------------------------------------------------------------------------------------------------
\section{The fully discrete problem: a symplectic MFD method}
\label{sec:5}

%\subsection{Convergence for discrete problem}
%\label{sub:5.1}

In this section we will derive  a  {\bf symplectic mimetic finite difference method} by applying a time integrator scheme to the semi-discrete problem  \eqref{eq:onde semi}. 
Because of the Hamiltonian structure of the system  \eqref{eq:onde semi} a symplectic scheme is usually employed to integrate in time, in order to preserve the symplectic structure of the flow map of the system.
% (see \cite{MR2132573} for more detail).
%We observed %above that the problem \eqref{eq:onde semi} is an Hamiltonian system, i.e. has the form
%\[
%z_t = \mathcal{J} \nabla_z H(z).
%\end{displaymath}
%As well known (see \cite{MR2132573} for more detail), the flow map $\Psi_{t, H}\colon \R^{2d} \to \R^{2d}$ of an Hamiltonian system is %a \textbf{symplectic map}, i.e. verifies
%\[
%(\partial_z \, \Psi_{t, H}(z))^T \mathcal{J}^{-1}(\partial_z \, \Psi_{t, H}(z)) = \mathcal{J}^{-1} \qquad \text{for all $z$.}
%\end{displaymath}
%An integrator method is \text{symplectic method} if the map associated is symplectic. We want now to use a symplectic method coupled %with MFD space discretisation in order to construct a full discrete scheme for equation \eqref{eq:ondesym}. Here we consider the %\textbf{symplectic implicit midpoint} (SIM) for \eqref{eq:onde semi} defined by
%\begin{equation}
%\label{eq:midpoint}
%z^{n+1} = z^n + \tau \, \mathcal{J} \nabla_z H(z^{n+ \frac{1}{2}}) \qquad z^{n+ \frac{1}{2}} = \frac{z^{n+1} + z^n}{2}.
%\end{equation}
%The SIM can be seen as a composition of two methods, the first one implicit, the second one explicit:
%\[
%z^{n+ \frac{1}{2}} = z^n + \frac{\tau}{2} \, \mathcal{J}\nabla_z H(z^{n+ \frac{1}{2}}), \qquad 
%z^n = z^{n+ \frac{1}{2}} + \frac{\tau}{2} \, \mathcal{J}\nabla_z H(z^{n+ \frac{1}{2}}).
%\end{displaymath}

Thus,  we apply the \textbf{symplectic implicit midpoint} (SIM) (which is a scheme of second order in time) to problem  \eqref{eq:onde semi} and get:
\begin{equation}
\label{eq:onde discreto}
\left \{
\begin{aligned}
& u_h^{n+1} = u_h^n + \tau \frac{v_h^{n+1} + v_h^n}{2} 
\\
& v_h^{n+1} = v_h^n + \tau \left(\Delta_h\, \frac{u_h^{n+1} + u_h^n}{2} - f'\left( \frac{u_h^{n+1} + u_h^n}{2}\right) \right) 
\\
& u_h^0 = u_{h,0}, \qquad v_h^0 = v_{h,0} 
\end{aligned}
\right .
\end{equation}
or equivalently
\begin{displaymath}
\left \{
\begin{aligned}
& v_h^{n+1} = v_h^n + \tau \left( \Delta_h\, \left(u_h^n + \tau \frac{v_h^{n+1} + v_h^n}{4} \right) - f'\left( u_h^n + \tau \frac{v_h^{n+1} + v_h^n}{4}\right) \right) 
\\
& u_h^{n+1} = u_h^n + \tau \frac{v_h^{n+1} + v_h^n}{2} 
\\
& u_h^0 = u_{h,0}, \qquad v_h^0 = v_{h,0} \ ,
\end{aligned}
\right . 
\end{displaymath}
where $(u^n_h,v^n_h)$ denotes the numerical approximation of $(u_h(t),v_h(t))$ at time $t_n=n\tau$,  for $n=0,\dots, N$ and $\tau = T/N$ represents  the time step length. Finally, by eliminating $v_h^n$ and $v_h^{n+1} $, we can express the system in the following form:
\begin{equation}
\label{eq:derivata seconda}
\begin{split}
\frac{u_h^{n+1} - 2 u_h^n + u_h^{n-1}}{\tau^2} &= \frac{v_h^{n+1} + v_h^n}{2 \tau} - \frac{v_h^{n} + v_h^{n-1}}{2 \tau} = \\
& = \frac{1}{2} \left( \Delta_h\, \left(\frac{u_h^{n+1} + 2 u_h^n + u_h^{n-1}}{2}\right) - \left(  f'\left(u_h^{n+\frac{1}{2}}\right) + f'\left(u_h^{n-\frac{1}{2}}\right)\right) \right) 
\end{split}
\end{equation}
where we use the notation
\begin{displaymath}
u_h^{q + \frac{1}{2}} = \frac{u_h^{q +1} + u_h^{q}}{2} \qquad \text{for $q=0,\dots, N-1$.}
\end{displaymath}
 
\subsection{Convergence for the fully discrete problem} 
\label{sub:5.1} 
 
We investigate the convergence of the sequence $\{u_h^n\}_{n = 1, \dots, N}$ to the exact solution $u$ of problem  \eqref{eq:onde}. The following result states the convergence of the numerical procedure in discrete $L^2$ norm.
% SIM integrator coupled with MFD space discretisation.

\begin{teorema}
\label{them:convergenza full}
Let $u$ be the solution of problem  \eqref{eq:onde} and let $\{u_h^n\}_{n = 1, \dots, N}$ be the sequence generated by  \eqref{eq:onde discreto}. Then, if $u \in C^3([0,T],H^2(\Omega))$, it follows  that:
\begin{equation}
\label{eq:stima discreta}
\|u(t_n)^I - u_h^n\|_{\mathcal{C}_h} \leq C \,  h^2 \left( |u_0|_{H^2(\Omega)} + |u_t|_{H^1(0, t_n, H^2(\Omega))}\right) + C \, \tau^2
\end{equation}
where the constant $C$ depends on the regularity of $u$.
\end{teorema}
\begin{proof}
Let us split the discrete error in the usual form 
\begin{equation}
\label{eq:thetarhodis}
u_h^n -u(t_n)^I = \left(u_h^n - \mathcal{P}_h \, u(t_n) \right) + \left( \mathcal{P}_h \, u(t_n)- u(t_n)^I \right) = \sigma^n + \rho^n.
\end{equation}
From  Lemma \ref{lemma1},  using the same argument in  \eqref{eq:rho}, we get
\begin{equation}
\label{eq:rhodis}
\begin{split}
\| \rho^n \|_{\mathcal{C}_h} &= \| \mathcal{P}_h \, u(t_n)- u(t_n)^I \|_{\mathcal{C}_h} \leq C\, h^2 \left( |u_0|_{H^2(\Omega)} + |u_t|_{H^1(0, t_n, H^2(\Omega))}\right).
\end{split}
\end{equation}
The analysis of the term $\sigma^n$ is more involved. We start by considering the first time step $n=0$ and we observe that using again   Lemma \ref{lemma1} it holds that  
\begin{equation}
\label{eq:sigmadis0}
\|\sigma^0\|_{\mathcal{C}_h} = \|u_h^0 - \mathcal{P}_h \, u(0)\|_{\mathcal{C}_h} = \|u_{h,0} - \mathcal{P}_h \, u_0\|_{\mathcal{C}_h} \leq C \, h^2 \, |u_0|_{H^2(\Omega)}.
\end{equation}
Let us analyse the first time step $t_1=\tau$. Using the regularity assumptions on  the solution $u$ in the time variable, we have that for all $x \in \Omega$ it holds 
\begin{equation}
\label{eq:utau}
\begin{split}
u(x,\tau) &=  u_0 + \tau \, v_0(x) + \frac{\tau^2}{2} \, u_{tt}\left(x, \frac{\tau}{2}\right) +R \\
&= u_0 + \tau \, v_0(x) + \frac{\tau^2}{2} \, \left({\rm div} \K \nabla \left(u\left(x, \frac{\tau}{2}\right)\right) - f'\left( x, \frac{\tau}{2}\right)  \right)+R
\end{split}
\end{equation}
where $R=O(\tau^3)$ is the rest in the Taylor expansion of $u(x, \cdot)$. By definition   \eqref{eq:onde discreto} with $n=0$,
\begin{displaymath}
u_h^1 = u_{h,0} + \tau \, v_{h,0} + \frac{\tau^2}{2}  \left( \Delta_h\, u_h^{1/2} - f'\left(u_h^{1/2}\right)\right).
\end{displaymath}
Then, using  \eqref{eq:lapint}, recalling that $u$ is the solution of  \eqref{eq:onde} and interpolating  \eqref{eq:utau} in $\mathcal{C}_h$ it follows that:
\begin{equation*}
\begin{split}
2&\sigma^{1/2} - \frac{\tau^2}{2}\Delta_h\, \sigma^{1/2} = (u_h^1 - \mathcal{P}_h \, u(\tau)) + \sigma^0  - \frac{\tau^2}{2} \Delta_h \left( u_h^{1/2} - \mathcal{P}_h \, u(\tau/2)\right)\\ 
&=u_{h,0} + \tau \, v_{h,0} - \frac{\tau^2}{2} f'\left(u_h^{1/2}\right) - \mathcal{P}_h\, u(\tau) +  \frac{\tau^2}{2} \bigl( {\rm div} \K \nabla \left(u\left(\tau/2\right)\right) \bigr)^I+ \sigma^0 \\
& = u_{h,0} + \tau \, v_{h,0} - \frac{\tau^2}{2} f'\left(u_h^{1/2}\right) - \mathcal{P}_h\, u(\tau) +  \frac{\tau^2}{2} \bigl( u_{tt}\left(\tau/2\right) + f'\left(u\left(\tau/2\right)\right) \bigr)^I+ \sigma^0 \\
& = u_{h,0} + \tau \, v_{h,0} + \frac{\tau^2}{2} \left(f'\left(u\left(\tau/2\right)\right)^I -  f'\left(u_h^{1/2}\right) \right) - \mathcal{P}_h\, u(\tau) + \bigl( u(\tau) - u_0 - \tau \, v_0 -R\bigr)^I  + \sigma^0 \\
& = \frac{\tau^2}{2} \left(f'\left(u\left(\tau/2\right)\right)^I -  f'\left(u_h^{1/2}\right) \right) - (\mathcal{P}_h\, u(\tau) - u(\tau)^I) - R^I + \sigma^0. 
\end{split}
\end{equation*}
Let us compute the scalar product  of both sides of the previous equation with $\sigma^{1/2}$, obtaining
\begin{equation}
\label{eq:utile}
\begin{split}
2 \left[\sigma^{1/2}, \sigma^{1/2} \right]_{\mathcal{C}_h} &- \frac{\tau^2}{2} \left[ \Delta_h \, \sigma^{1/2}, \sigma^{1/2} \right]_{\mathcal{C}_h}=\\
& = \left[ \frac{\tau^2}{2} \left(f'\left(u\left(\tau/2\right)\right)^I -  f'\left(u_h^{1/2}\right) \right) -\bigl( \mathcal{P}_h\, u(\tau) - u(\tau)^I \bigr) - R^I + \sigma^0,\sigma^{1/2} \right]_{\mathcal{C}_h}  \\
& \leq C  \left( \frac{\tau^2}{2}\left \| u\left(\tau/2\right) - u_h^{1/2} \right\|_{\mathcal{C}_h}  + \left \|\rho^1 \right\|_{\mathcal{C}_h} + \|\sigma^0\|_{\mathcal{C}_h} + \tau^3 \right) \left \|\sigma^{1/2} \right\|_{\mathcal{C}_h} \\
& \leq C  \left( \frac{\tau^2}{2}\left \|\sigma^{1/2} \right\|_{\mathcal{C}_h} + \frac{\tau^2}{2}\left \|\rho^{1/2} \right\|_{\mathcal{C}_h} + \left \|\rho^1 \right\|_{\mathcal{C}_h} + \|\sigma^0\|_{\mathcal{C}_h} + \tau^3 \right) \left \|\sigma^{1/2} \right\|_{\mathcal{C}_h} \ .\\
\end{split}
\end{equation}
Now, since from  Lemma \ref{lemma2} it follows $- [\Delta_h \,v_h, v_h]_{\mathcal{C}_h} \geq 0$ for all $v_h$,  for small values of $\tau$ using  \eqref{eq:rhodis} and  \eqref{eq:sigmadis0}, we get
\begin{displaymath}
\left \|\sigma^{1/2}\right\|_{\mathcal{C}_h} \leq C \, h^2 \biggl( |u_0|_{H^2(\Omega)} + |u_t|_{H^1(0, \tau, H^2(\Omega))}  \biggr) \ ,
\end{displaymath}
and we can conclude that
\begin{equation}
\label{eq:sigmadis1}
\|\sigma^1\|_{\mathcal{C}_h} \leq 2 \left \|\sigma^{1/2}\right\|_{\mathcal{C}_h} +  \|\sigma^0\|_{\mathcal{C}_h} \leq C \, h^2 \biggl(  |u_0|_{H^2(\Omega)} + |u_t|_{H^1(0, \tau, H^2(\Omega))} \biggr).
\end{equation}
Now, we bound the error for a  general time step $n \geq 1$. It is easy to see that the following relations hold
%\label{eq:relations}
\begin{gather*}
\frac{u(t_{n+1}) -2 u(t_n) + u(t_{n-1})}{\tau^2} = u_{tt}(t_n) +  \bar{R}  \ , \\
u(t_{n+1}) + 2 u(t_n) + u(t_{n-1}) = 4 u(t_n) +   \bar{R} \ ,\\
\frac{f'(u(t_{n + 1})) + f'(u(t_{n})) }{2} =  f' \left( u \left( t_{n + \frac{1}{2}} \right) \right) +  \bar{R} \ ,
\end{gather*}
where $\bar{R} = O(\tau^2)$ denotes the general rests in the Taylor expansion. Using the previous Taylor expansions, the definition of the scheme  \eqref{eq:derivata seconda} and  \eqref{eq:lapint}, and recalling that $u$ is the solution of  \eqref{eq:onde}, we have
\begin{equation}
\label{eq:sigman}
\begin{split}
& \frac{\sigma^{n+1} - 2 \sigma^n + \sigma^{n-1} }{\tau^2} - \frac{1}{2}\, \Delta_h\left( \frac{\sigma^{n+1} + 2 \sigma^n +\sigma^{n-1}}{2} \right) \\
& \quad = -\frac{1}{2} \left(f'\left(u_h^{n+\frac{1}{2}}\right) + f'\left(u_h^{n-\frac{1}{2}}\right)\right) - \mathcal{P}_h \left(\frac{u(t_{n+1}) - 2u(t_n) + u(t_{n-1})}{\tau^2} \right) +\\
& \quad \quad + \, \left(  {\rm div} \K \nabla  \left( \frac{ u(t_{n+1}) + 2u(t_n) + u(t_{n-1})}{4} \right)\right)^I = \\ 
& \quad = -\frac{1}{2} \left(f'\left(u_h^{n+\frac{1}{2}}\right) + f'\left(u_h^{n-\frac{1}{2}}\right)\right) +   \,\left(\frac{f'(u(t_{n+1})) + 2 f'(u(t_n)) + f'(u(t_{n-1}))}{4}\right)^I + \\
& \quad \quad +  \, \left( \frac{u_{tt}(t_{n+1}) + 2u_{tt}(t_n) + u_{tt}(t_{n-1})}{4} \right)^I - \mathcal{P}_h \left(\frac{u(t_{n+1}) - 2u(t_n) + u(t_{n-1})}{\tau^2} \right)\\
& \quad = \frac{1}{2} \left( f' \left( u \left( t_{n+\frac{1}{2}} \right) \right)^I  - f'\left(u_h^{n+\frac{1}{2}}\right) \right) + \frac{1}{2} \left( f' \left( u \left(t_{n-\frac{1}{2}} \right) \right)^I  - f'\left(u_h^{n-\frac{1}{2}}\right)\right) + \\
& \quad \quad +  \, \left(  u_{tt} (t_n)^I - \mathcal{P}_h \, u_{tt}(t_n)  \right) + \, \bar{R}\\
& \quad=    \alpha^{n+\frac{1}{2}} + \alpha^{n-\frac{1}{2}} -  \rho_{tt}^n + \bar{R}  \ ,
\end{split}
\end{equation}
where $\alpha^q = \frac{f'(u(t_q))^I - f'(u_h^q)}{2}$ with $q= n \pm \frac{1}{2}$. 

Now, let
\begin{displaymath}
\delta^{n+ \frac{1}{2}} := \frac{\sigma^{n+1} - \sigma^n}{\tau}
\end{displaymath}
and let us observe that the following relations hold:
\begin{equation}
\label{eq:sigmadelta}
\frac{\sigma^{n+1} - 2 \sigma^n + \sigma^{n-1}}{\tau^2} = \frac{\delta^{n+\frac{1}{2}}- \delta^{n-\frac{1}{2}}}{\tau}\ , \qquad  \delta^{n+ \frac{1}{2}} + \delta^{n- \frac{1}{2}} = 2\, \frac{\sigma^{n+ \frac{1}{2}} - \sigma^{n- \frac{1}{2}}}{\tau}. 
\end{equation}
Let us make the inner product of both sides of  \eqref{eq:sigman} with $\delta^{n+ \frac{1}{2}} + \delta^{n- \frac{1}{2}}$.   For the first term of the left-hand side, using  \eqref{eq:sigmadelta}, we get
\begin{equation}
\label{eq:1l}
\frac{1}{\tau}\, \left [ \delta^{n+\frac{1}{2}}- \delta^{n-\frac{1}{2}} , \delta^{n+ \frac{1}{2}} + \delta^{n- \frac{1}{2}} \right]_{\mathcal{C}_h} = \frac{1}{\tau} \left(\left \|\delta^{n+\frac{1}{2}} \right \|^2_{\mathcal{C}_h} - \left \|\delta^{n-\frac{1}{2}} \right\|^2_{\mathcal{C}_h} \right).
\end{equation}
For the second term of the left-hand side in  \eqref{eq:sigman}, using  \eqref{eq:sigmadelta} and since $\Delta_h$ is self-adjoint, we have
\begin{multline}
\label{eq:2l}
-\frac{1}{\tau} \, \left[\Delta_h \left(\sigma^{n+\frac{1}{2}} + \sigma^{n-\frac{1}{2}}\right), \sigma^{n+ \frac{1}{2}} - \sigma^{n- \frac{1}{2}} \right]_{\mathcal{C}_h} =\\ 
= - \frac{1}{\tau} \, \left(  \left [\Delta_h\,\sigma^{n + \frac{1}{2}}, \sigma^{n + \frac{1}{2}} \right]_{\mathcal{C}_h} - \left [\Delta_h\, \sigma^{n - \frac{1}{2}}, \sigma^{n - \frac{1}{2}} \right]_{\mathcal{C}_h}\right).
\end{multline}
To bound the right-hand side, we preliminary observe that using the same argument of the proof of  Theorem \ref{thm:convergenza} and  Lemma \ref{lemma3},  it holds that
\begin{displaymath}
\left \|  \alpha^q \right \|_{\mathcal{C}_h} = \frac{1}{2}\left \|  f'(u(t_q))^I - f'(u_h^q)\right \|_{\mathcal{C}_h} \leq C \left \|  u(t_q) - u_h^q \right \|_{\mathcal{C}_h} \leq  C \left( \left \|  \sigma^q \right \|_{\mathcal{C}_h}  + \left \|  \rho^q \right \|_{\mathcal{C}_h}  \right).
\end{displaymath}
Therefore, using the previous bound, the Cauchy-Schwartz inequality and the usual estimate in $\rho$,  we derive
\begin{equation}
\label{eq:right}
\begin{split}
&\left [  \alpha^{n+\frac{1}{2}} + \alpha^{n-\frac{1}{2}} + \rho_{tt}^n + \bar{R} , \delta^{n+ \frac{1}{2}} + \delta^{n- \frac{1}{2}}\right]_{\mathcal{C}_h} \\
&\leq C \left(  \left \|\alpha^{n+\frac{1}{2}}\right\|_{\mathcal{C}_h} +\left \| \alpha^{n-\frac{1}{2}}\right\|_{\mathcal{C}_h} + \left \|\rho_{tt}^n\right\|_{\mathcal{C}_h} + \tau^2 \right) \left\|\delta^{n+\frac{1}{2}} + \delta^{n-\frac{1}{2}} \right\|_{\mathcal{C}_h}  \\
&\leq C \left(  \left \|\alpha^{n+\frac{1}{2}}\right\|^2_{\mathcal{C}_h} +\left \| \alpha^{n-\frac{1}{2}}\right\|^2_{\mathcal{C}_h} + \left \|\rho_{tt}^n\right\|^2_{\mathcal{C}_h}  + \left\|\delta^{n+\frac{1}{2}} + \delta^{n-\frac{1}{2}} \right\|^2_{\mathcal{C}_h} + \tau^4 \right) \\
&\leq C \left( \left \|\sigma^{n+\frac{1}{2}}\right\|^2_{\mathcal{C}_h} +\left \| \sigma^{n-\frac{1}{2}}\right\|^2_{\mathcal{C}_h} + \left \|\rho^{n+\frac{1}{2}}\right\|^2_{\mathcal{C}_h} +\left \| \rho^{n-\frac{1}{2}}\right\|^2_{\mathcal{C}_h} +\left \|\rho_{tt}^n\right\|^2_{\mathcal{C}_h}  +  \left\|\delta^{n+\frac{1}{2}} + \delta^{n-\frac{1}{2}} \right\|^2_{\mathcal{C}_h} + \tau^4\right) \\ 
&\leq C \left( \left \|\sigma^{n + 1}\right\|^2_{\mathcal{C}_h} + 2 \left \| \sigma^{n}\right\|^2_{\mathcal{C}_h} + \left \| \sigma^{n-1}\right\|^2_{\mathcal{C}_h}   +  \left\|\delta^{n+\frac{1}{2}} \right\|^2_{\mathcal{C}_h} + \left \|\delta^{n-\frac{1}{2}} \right\|^2_{\mathcal{C}_h} + \tau^4 + h^4 \right).
\end{split}
\end{equation}
Collecting  \eqref{eq:1l},  \eqref{eq:2l} and  \eqref{eq:right} in  \eqref{eq:sigman}, we obtain
\begin{multline}
\label{eq:leftright}
\frac{1}{\tau} \, \left( \left \|\delta^{n+\frac{1}{2}} \right \|^2_{\mathcal{C}_h} -\left \|\delta^{n-\frac{1}{2}} \right\|^2_{\mathcal{C}_h} -  \left [\Delta_h\,\sigma^{n + \frac{1}{2}}, \sigma^{n + \frac{1}{2}} \right]_{\mathcal{C}_h} + \left [\Delta_h\, \sigma^{n - \frac{1}{2}}, \sigma^{n - \frac{1}{2}} \right]_{\mathcal{C}_h} \right)\\
\leq C  \, \left( \left \|\sigma^{n + 1}\right\|^2_{\mathcal{C}_h} + 2 \left \| \sigma^{n}\right\|^2_{\mathcal{C}_h} + \left \| \sigma^{n-1}\right\|^2_{\mathcal{C}_h}   +  \left\|\delta^{n+\frac{1}{2}} \right\|^2_{\mathcal{C}_h} + \left \|\delta^{n-\frac{1}{2}} \right\|^2_{\mathcal{C}_h} + \tau^4 + h^4 \right).
\end{multline}
Moreover,  \eqref{eq:sigmadelta} and some simple calculations give:
\begin{equation}
\label{eq:trick}
\begin{split}
\frac{1}{\tau} \, \left( \|\sigma^{n+1}\|^2_{\mathcal{C}_h}  - \|\sigma^{n-1}\|^2_{\mathcal{C}_h}  \right) &= \left[\sigma^{n+1} + \sigma^{n-1}, \delta^{n+ \frac{1}{2}} + \delta^{n- \frac{1}{2}} \right] \\
& \leq C \left( \|\sigma^{n+1}\|^2_{\mathcal{C}_h} + 2\|\sigma^n\|^2_{\mathcal{C}_h}  + \|\sigma^{n-1}\|^2_{\mathcal{C}_h} + \left\|\delta^{n+ \frac{1}{2}} \right \|^2_{\mathcal{C}_h} + \left\|\delta^{n- \frac{1}{2}} \right \|^2_{\mathcal{C}_h} \right).
\end{split}
\end{equation}
Now let us define
\begin{displaymath}
\Gamma^n := \left\|\delta^{n+ \frac{1}{2}} \right \|^2_{\mathcal{C}_h} +  \|\sigma^{n+1}\|^2_{\mathcal{C}_h} + \|\sigma^n\|^2_{\mathcal{C}_h} - \left[ \Delta_h\, \sigma^{n+\frac{1}{2}} , \sigma^{n+ \frac{1}{2}}  \right]_{\mathcal{C}_h}\ .
\end{displaymath}
Using  the estimates  \eqref{eq:leftright} and  \eqref{eq:trick}, recalling that the operator $-\Delta_h$ is positive definite, we derive that
\begin{displaymath}
\frac{\Gamma^n - \Gamma^{n-1}}{\tau} \leq C\, (h^2 + \tau^2)^2  + C \,(\Gamma^n + \Gamma^{n-1})\ ,
\end{displaymath}
and, by using the discrete Gronwall inequality,  we obtain
\begin{displaymath}
\Gamma^n \leq \left( \Gamma^0 + \sum_{k=1}^n \tau \, (h^2 + \tau^2)^2 \right) e^{\tau \, 4C t_n} \ .
\end{displaymath} 
Now, using analogous arguments in  \eqref{eq:utile} and recalling bounds  \eqref{eq:sigmadis0},  \eqref{eq:sigmadis1}, we obtain
\begin{displaymath}
\Gamma^0 \leq C(h^2 + \tau^2)^2,
\end{displaymath}
and thus
\begin{displaymath}
\|\sigma^n\|^2_{\mathcal{C}_h} \leq \Gamma^n \leq C(h^2 + \tau^2)^2  e^{\tau \, 4C t_n}.
\end{displaymath}
Hence, since $t_n \leq T$, the above bounds gives
\begin{equation}
\label{eq:sigmafinal}
\|\sigma^n\|_{\mathcal{C}_h} \leq C(h^2 + \tau^2)\ ,
\end{equation}
for all $n=1,\ldots,N$, and collecting  \eqref{eq:rhodis} and  \eqref{eq:sigmafinal} in  \eqref{eq:thetarhodis} we get the thesis.
\end{proof}

\subsection{Conservation laws for the fully discrete problem}
\label{sub:5.2}

The following result shows how the fully discrete method, built combining  the MFD method and the symplectic implicit midpoint scheme,  preserves, within an order $\tau^2$ of approximation, the Hamiltonian functional. Using classical results on the symplectic integrator methosd (see for instance \cite{MR1904823}) we can state the following theorem about the long time stability of the Hamiltonian.

\begin{teorema}
\label{thm:conservazione ham}
Let $\left(u_h^n, v_h^n \right)$ be the sequence generated by system  \eqref{eq:onde discreto}. Then if $T \leq e^{\gamma/\tau}\,\tau^2$, for a suitable positive constant $\gamma$, it holds that: 
\begin{equation}
\label{eq:conservazione ham}
\left | \mathcal{H}_h\left[u_h^N, v_h^N \right] -  \mathcal{H}_h\left[u_h^0, v_h^0  \right]\right| \leq C \, \tau^2 \ .
\end{equation}
\end{teorema}

\begin{osservazione}
\label{oss:esatta}
It is well known that the SIM preserves the quadratic first integrals. Therefore if the load term $f$ is quadratic, i.e. $f(s) = ks^2$, with $k$ constant the Hamiltonian is exactly preserved along the solutions.
\end{osservazione}

By collecting the estimates of  Theorem \ref{thm:conservazione1} and Theorem \ref{thm:conservazione ham} we can provide a bound for the  error in the Hamiltonian of the fully discrete procedure, stemming from the MFD discretization in  space and the SIM integration in time.

\begin{teorema}
Let $(u(t), v(t))$ be the solution of problem  \eqref{eq:ondesym} and let $\left(u_h^n, v_h^n \right)$ be the sequence generated by system  \eqref{eq:onde discreto}. Then, it follows that
\begin{equation}
\label{eq:conservazione total ham}
\left | \mathcal{H}_h\left[u_h^N, v_h^N \right] -  \mathcal{H}\left[u(t_N), v(t_N) \right]\right| \leq C  ( \tau^2 \  + h^2).
\end{equation}
\end{teorema}

In  Section \ref{sec:3} we have introduced the semi-discrete Energy density conservation law.  Now we analyse the effect of time discretization in the semi-discrete Energy density conservation law.

\begin{teorema}
\label{thm: discrete energy conservation}
Let $\left(u_h^n, v_h^n \right)$ be the sequence generated by system  \eqref{eq:onde discreto}, and let for all $c \in \mathcal{T}_h$ and for all $n$
\begin{displaymath}
E_{h,c}(u_h^n, v_h^n) := \frac{1}{2} |c| \left(v^n_{h,c} \right)^2 + \frac{1}{2} \left [({\gd} \,u_h^n)_c , ({\gd} \,u_h^n)_c \right]_{\mathcal{F}_{h,c}} + |c|\,f(u^n_{h,c})
\end{displaymath}
and
\begin{displaymath}
F_{h,c}\left (u_h^{n }, v_h^{n } \right) = - |c|\left (\Delta_h\,u^{n }_{h} \right)_c \, v^{n }_{h,c} -  \left[\left({\gd}\, v^{n }_h \right)_c, 	\left({\gd}\, u^{n }_h \right)_c \right]_{\mathcal{F}_{h,c}}.
\end{displaymath}
Then, the following estimate holds for all $n$
\begin{equation}
\label{eq:discrete energy conservation}
\left | \frac{E_{h,c}(u_h^{n+1}, v_h^{n+1}) - E_{h,c}(u_h^n, v_h^n)} {\tau} + F_{h,c}\left (u_h^{n + \frac{1}{2}}, v_h^{n + \frac{1}{2}} \right) \right| \leq C |c|\, \tau^2 \ .
\end{equation}
\end{teorema}

\begin{proof}
We observe that, using  \eqref{eq:onde discreto}, it follows that
\begin{multline}
\label{eq:discrete cons1}
\frac{1}{2 \tau} \left( \left [({\gd} \,u_h^{n+1})_c , ({\gd} \,u_h^{n+1})_c \right]_{\mathcal{F}_{h,c}} -  \left [({\gd} \,u_h^n)_c , ({\gd} \,u_h^n)_c \right]_{\mathcal{F}_{h,c}}\right)= \\
= \frac{1}{2 \tau}\left(\left [({\gd} \,(u_h^{n+1} - u_h^n))_c , ({\gd} \,(u_h^{n+1} + u_h^n))_c \right]_{\mathcal{F}_{h,c}}\right) =\\
= \left[\left({\gd}\, v^{n + \frac{1}{2}}_h \right)_c, 	\left({\gd}\, u^{n + \frac{1}{2}}_h \right)_c \right]_{\mathcal{F}_{h,c}}.
\end{multline}
and 
\begin{equation}
\label{eq:discrete cons2}
\begin{split}
\frac{\left(v^{n+1}_{h,c} \right)^2 - \left(v^n_{h,c} \right)^2}{2 \tau} &=v_{h,c}^{n+ \frac{1}{2}}  \frac{v^{n+1}_{h,c} - v^{n}_{h,c}}{\tau} =\\
& = v_{h,c}^{n+ \frac{1}{2}} \left( \Delta_h\, u_h^{n + \frac{1}{2}} \right)_c  - v_{h,c}^{n+ \frac{1}{2}} f'\left( u_{h,c}^{n+\frac{1}{2}}\right).
\end{split}
\end{equation}
Therefore, by collecting  \eqref{eq:discrete cons1} and  \eqref{eq:discrete cons2}, we get
\begin{multline}
\label{eq:discrete con3}
\frac{E_{h,c}(u_h^{n+1}, v_h^{n+1}) - E_{h,c}(u_h^n, v_h^n)} {\tau} +  F_{h,c}\left (u_h^{n + \frac{1}{2}}, v_h^{n + \frac{1}{2}} \right) =\\ 
= |c| \left( \frac{f(u_{h,c}^{n+1}) - f(u_{h,c}^n)}{\tau} - v_{h,c}^{n+ \frac{1}{2}} f'\left( u_{h,c}^{n+\frac{1}{2}}\right)\right) .
\end{multline}
Now, by the Taylor expansion, we derive
\begin{displaymath}
f(u_{h,c}^{n+1}) - f(u_{h,c}^n) = (u_{h,c}^{n+1} - u_{h,c}^n) f'\left( u_{h,c}^{n+\frac{1}{2}}\right) + R \tau^3 = \tau\, v_{h,c}^{n+ \frac{1}{2}}  f'\left( u_{h,c}^{n+\frac{1}{2}}\right) + R
\end{displaymath}
where $R = O(\tau^3)$ denotes the rest, and thus
\begin{displaymath}
\frac{E_{h,c}(u_h^{n+1}, v_h^{n+1}) - E_{h,c}(u_h^n, v_h^n)} {\tau} +  F_{h,c}\left (u_h^{n + \frac{1}{2}}, v_h^{n + \frac{1}{2}} \right) = R |c| \tau^2.
\end{displaymath}
\end{proof}

%\begin{osservazione} DIRE QUALCOSA SULLA CONSERVAZIONE DEGLI ALTRI INVARIANTI. SE NON RIUSCIAMO A DIMOSTRARLA %TEORICAMENTE, VEDERE SE ALMENO LE PROVE NUMERICHE SONO BUONE. 
%\end{osservazione}

% --------------------------------------------------------------------------------------------------------------------------
% --------------------------------------------------------------------------------------------------------------------------

\section{Numerical tests}
\label{sec:6}

In the present section we present some numerical results for the fully discrete case, i.e. SIM coupled with the MFD spatial discretization.  The convergence of MFD has  been evaluated in the discrete relative $L^2(\Omega)$ norm of the difference between the interpolant $u^I \in \mathcal{C}_h$ of the exact solution $u$ and the numerical solution $u_h$ at the final time $T$, i.e.
\begin{displaymath}
E_{h,\tau}:= \frac {\| u^I(T) - u_{h,N}\|_{\mathcal{C}_h}}{\|u^I(T)\|_{\mathcal{C}_h}}.
\end{displaymath}
Moreover we tested the total error in the Hamiltonian functional at the final step $N$, among the discrete solution and the continuous solution, that is:
\begin{displaymath}
\sigma_{h,\tau}:= \left | \mathcal{H}_h[u_h^N, v_h^N] - \mathcal{H}[u_0, v_0]\right|.
\end{displaymath}
We tested also the conservation of the Hamiltonian functional with respect to time integration, that is
\begin{displaymath}
\delta_{h,\tau}:= \left | \mathcal{H}_h[u_h^N, v_h^N] - \mathcal{H}_h[u_h^0, v_h^0]\right|
\end{displaymath}  
and the error in Energy density conservation law  that is:
\begin{displaymath}
\epsilon_{h,\tau}:= \max_{c \in \mathcal{T}_h} \, \left |\frac{E_{h,c}(u_h^{N}, v_h^{N}) - E_{h,c}(u_h^{N-1}, v_h^{N-1})} {\tau} + F_{h,c}\left (u_h^{N - \frac{1}{2}}, v_h^{N - \frac{1}{2}} \right) \right|.
\end{displaymath}

We have considered the spatial domain $\Omega = [0,1] \times [0,1] \subseteq \R^2$, and a general sequence of \textbf{Voronoi meshes} with $h=0.2, 0.1, 0.05, 0.025$ (see Figure \ref{fig:2}), and $\tau = 0.1, 0.05, 0.025, 0.0125$. For the generation of the Voronoi meshes we used the code Polymesher in \cite{TPPM12}.

\begin{figure}[!h]
\centering
\includegraphics[scale=0.25]{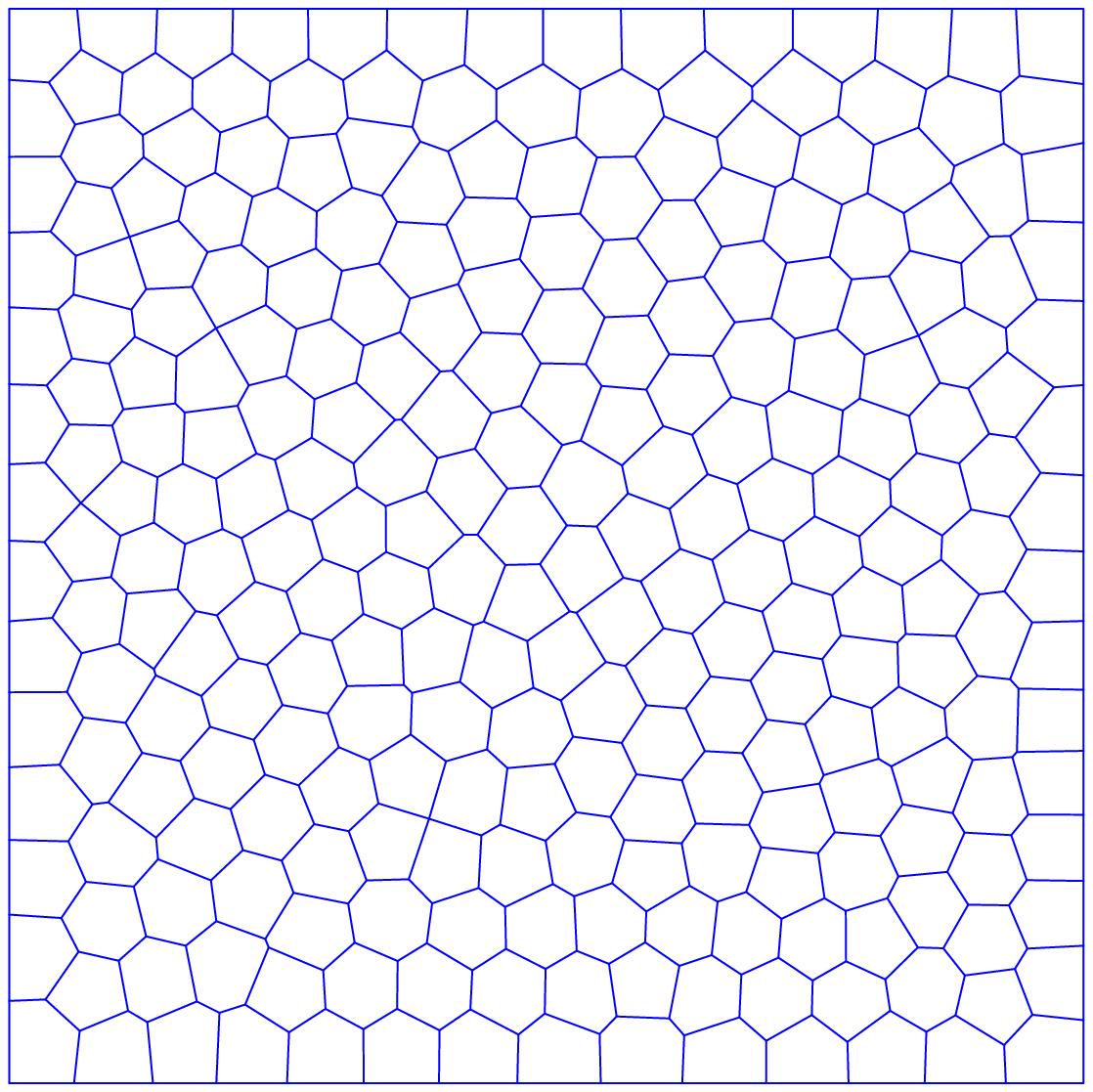}
\caption{Sequence of Voronoi mesh with $0.05$.}
\label{fig:2}
\end{figure}

\begin{test}
\label{test1}

We consider  problem  \eqref{eq:ondesym} with the material tensor $\K = I_2$ and the load term $f = \frac{1 - 2\pi^2}{2}u^2$, where the initial data $u_0$ and $v_0$  are chosen in accordance with the exact solution 
\begin{equation}
\label{soluzione1}
u(t,x_1, x_2) = \sin(t) \sin(\pi x_1)\sin(\pi x_2).
\end{equation}
We implement the fully discrete problem in the time interval $[0,1]$  with the SIM coupled with the MFD discretization for  the sequence of polygonal meshes introduced above. In  Table \ref{tab::1} we choose $\tau=0.001$ and we show the errors in the solution $E_{h, \tau}$, for the Hamiltonian $\sigma_{h, \tau}$ and for the Energy $\epsilon_{h, \tau}$ for different values of the mesh size $h$.

\begin{table}[!h]
\centering
\caption{$E_{h, \tau}$, $\delta_{h, \tau}$ and $\epsilon_{h, \tau}$ for fixed time step size $\tau=0.001$.}
\begin{tabular}{l*{4}{c}}
\toprule
& $h= 0.2$        & $h=0.1$      & $h=0.05$       & $h=0.025$       \\
\midrule
\multirow{1}*{$E_{h, \tau}$}
& $7.2713323e-01$ &  $1.9846010e-01$ &  $5.2502301e-02$ & $1.3086316e-02$\\
%\midrule
\multirow{1}*{$\sigma_{h, \tau}$}
& $7.5726618e-03$ &  $1.9485290e-03$ &  $5.0100939e-04$ & $1.2559774e-04$\\
%\midrule
\multirow{1}*{$\epsilon_{h, \tau}$}
& $2.6233230e-01$ &  $1.4264664e-02$ &  $1.5776230e-03$ & $4.1846647e-04$\\
\bottomrule
\end{tabular}
\label{tab:1}
\end{table}

In  Figure \ref{fig:3} we plot the asymptotic behaviour of the errors in the solution and Hamiltonian as a function of $h$, in accordance with the theoretical order of convergence $h^2$.
\begin{figure}[!h]
\centering
\includegraphics[scale=0.25]{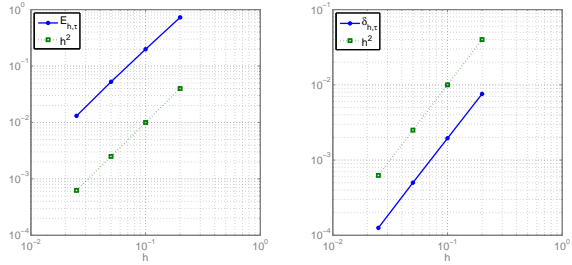}
\caption{Asymptotic behaviour of $E_{h,\tau}$ and $\delta_{h,\tau}$ as a function of $h$ for $\tau=0.001$.}
\label{fig:3}
\end{figure}

In  Figure \ref{fig:4} we plot the asymptotic behaviour of the errors in the Energy density conservation law $\epsilon_{h, \tau}$ at the final step $N$ as a function of $h$ for $\tau=h$. We observe that, using  \eqref{eq:discrete energy conservation}, we expect an order $h^4$ of convergence. 
\begin{figure}[!h]
\centering
\includegraphics[scale=0.25]{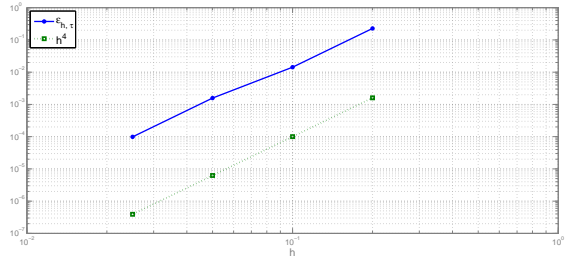}
\caption{Asymptotic behaviour of $\epsilon_{h,\tau}$ as a function of $h$ for $\tau=h$.}
\label{fig:4}
\end{figure}

In  Figure \ref{fig:5} we extend the time interval setting $T=100$ and  we show the behaviour of the error $\delta_{h,\tau}$ in the discrete Hamiltonian functional along the sequence $(u_h^n, v_h^n)$ with respect to the initial value  $(u_h^0, v_h^0)$ for $h=0.05$ and $\tau=0.001$. We can observe that the Hamiltonian is numerically preserved by SIM. This results is in accordance with  Remark \ref{oss:esatta}, indeed in the test we are considering a quadratic function $f(u)$.

\begin{figure}[!h]
\centering
\includegraphics[scale=0.25]{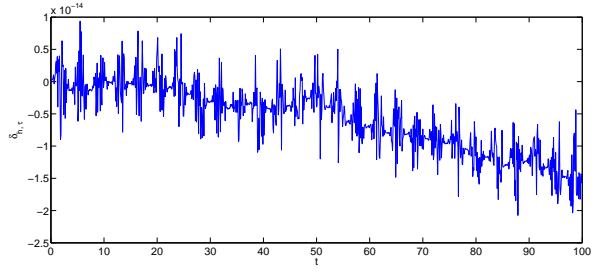}
\caption{Behaviour of discrete Hamiltonian functional along the sequence $(u_h^n, v_h^n)$ with $h=0.05$ and $\tau=0.001$.}
\label{fig:5}
\end{figure}
In  Figure \ref{fig:6} we consider as before $T=100$ and  we plot the evolution of the error in Energy conservation law along the discrete solution with $h=0.05$ and $\tau=0.001$.
\begin{figure}[!h]
\centering
\includegraphics[scale=0.25]{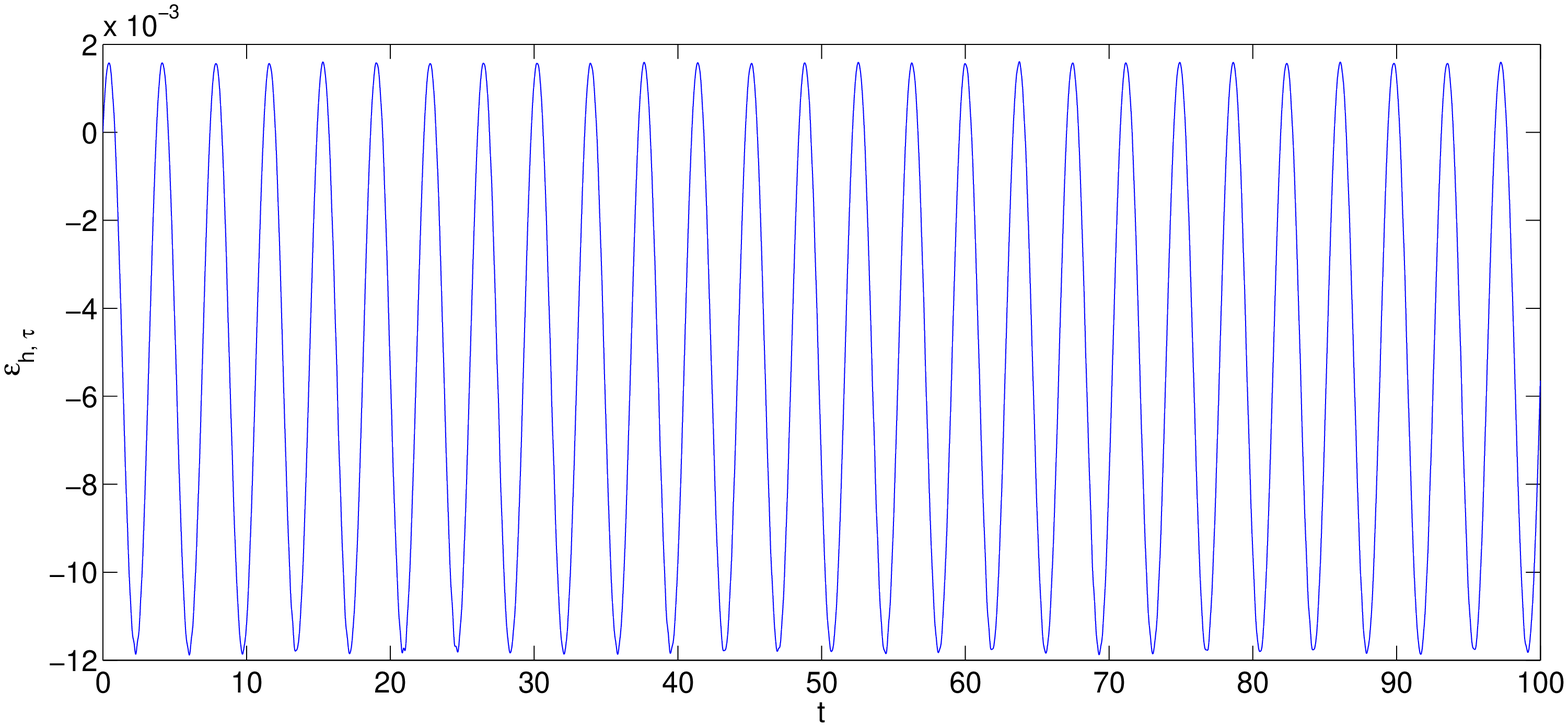}
\caption{Behaviour of Energy conservation law error along the sequence $(u_h^n, v_h^n)$ with $h=0.05$ and $\tau=0.001$.}
\label{fig:6}
\end{figure}

\end{test}

\begin{test}
\label{test2}
We consider problem  \eqref{eq:ondesym} with  material tensor,  load and initial data  given by
\begin{displaymath}
\K = I_2, \qquad f(u) = \sin(u), \qquad u_0(x) = 0, \qquad v_0(x) = \sin(\pi x_1) \sin(\pi x_2).
\end{displaymath}
We implement the fully discrete problem in the time interval $[0,1]$  with the SIM coupled with the MFD discretization for  the usual sequence of polygonal meshes introduced above. In  Table \ref{tab::2} we choose $\tau=0.001$ and we show the errors for the Hamiltonian $\sigma_{h, \tau}$ and for the Energy $\epsilon_{h, \tau}$ for different values of the mesh size $h$. We observe that we achieve the theoretical order $h^2$ of convergence.  In  Table \ref{tab::3} we fix the mesh size $h=0.05$ and we display the errors for the Hamiltonian and for the Energy as a function of $\tau$. In this case we observe that the error for the Hamiltonian is almost constant in $\tau$: the error due to the spatial discretization dominates the time component of the error.

\begin{table}[!h]
\centering
\caption{$\delta_{h, \tau}$ and $\epsilon_{h, \tau}$ for fixed time step size $\tau=0.001$.}
\begin{tabular}{l*{4}{c}}
\toprule
& $h= 0.2$        & $h=0.1$      & $h=0.05$       & $h=0.025$       \\
\midrule
\multirow{1}*{$\sigma_{h, \tau}$}
& $7.5726593e-03$ &  $1.9485268e-03$ &  $5.0100734e-04$ & $1.2486786e-04$\\
%\midrule
\multirow{1}*{$\epsilon_{h, \tau}$}
& $2.2022797e-02$ &  $7.5684283e-03$ &  $1.3466833e-03$ & $3.3563875e-04$\\
\bottomrule
\end{tabular}
\label{tab:2}
\end{table}

\begin{table}[!h]
\centering
\caption{$\delta_{h, \tau}$ and $\epsilon_{h, \tau}$ for fixed mesh diameter $h=0.05$.}
\begin{tabular}{l*{4}{c}}
\toprule
& $\tau= 0.1$        & $\tau=0.05$      & $\tau=0.025$       & $\tau=0.0125$       \\
\midrule
\multirow{1}*{$\sigma_{h, \tau}$}
& $4.8131147e-04$ &  $4.9593346e-04$ &  $4.997307e-04$ & $5.0068913e-04$\\
%\midrule
\multirow{1}*{$\epsilon_{h, \tau}$}
& $2.2022797e-02$ &  $7.5684283e-03$ &  $1.3466833e-03$ & $3.3563875e-04$\\
\bottomrule
\end{tabular}
\label{tab:3}
\end{table}

In  Figure \ref{fig:7} we consider  a larger final time $T=100$ and  we plot the behaviour of the error $\delta_{h,\tau}$ in the discrete Hamiltonian functional along the sequence $(u_h^n, v_h^n)$ with respect to the initial value  $(u_h^0, v_h^0)$ for $h=0.05$ and $\tau=0.001$.

\begin{figure}[!h]
\centering
\includegraphics[scale=0.25]{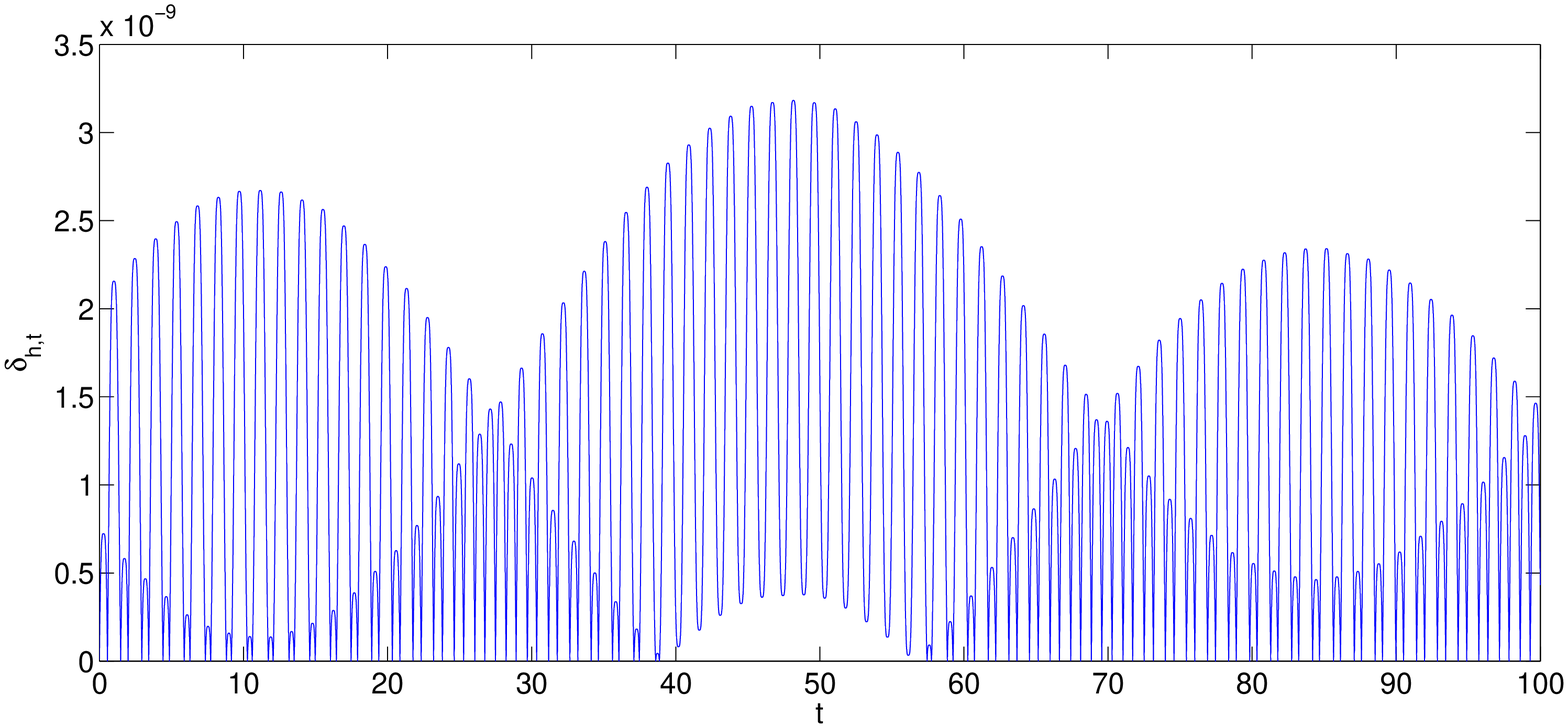}
\caption{Behaviour of discrete Hamiltonian functional along the sequence $(u_h^n, v_h^n)$ with $h=0.05$ and $\tau=0.001$.}
\label{fig:7}
\end{figure}

In  Figure \ref{fig:8} we consider again $T=100$ and  we show the evolution of the error in Energy conservation law along the discrete solution with $h=0.05$ and $\tau=0.001$.

\begin{figure}[!h]
\centering
\includegraphics[scale=0.25]{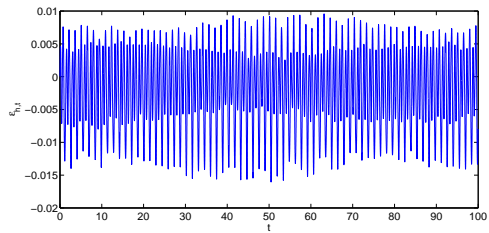}
\caption{Behaviour of Energy conservation law error along the sequence $(u_h^n, v_h^n)$ with $h=0.05$ and $\tau=0.001$.}
\label{fig:8}
\end{figure}

\end{test}

% --------------------------------------------------------------------------------------------------------------------------
% --------------------------------------------------------------------------------------------------------------------------

\section{Conclusions}
In this paper we have analysed  the structure and the time invariants of the nonlinear wave equation discretized by mimetic approach. We have proved that the MFD discretization preserves the hamiltonian formulation of the problem and that the Hamiltonian and the Energy are still semi-discrete invariants of the solution. We have also derived a convergence theory for the method, obtaining an $h^2$ order for the $L^2$ discrete norm of the error among the solution of the continuous  and discrete problems. We have then considered the fully discrete scheme by making use of the MFD method coupled with the SIM time integrator: we have derived the convergence rate of the method
and we have investigate the behaviour the  Hamitonian and Energy. 
In light of these results we belive that the spatial discretization by making use of the MFD technique
is a good choice in the context PDEs with conservation laws.
In the present manuscript the focus is on the spatial discretization, thus
the use of the mid-point scheme, for the time discretization, should
be simply understood as a model symplectic method. 
%Nevertheless, we would point out that the mid-point scheme could be replaced by a more sophisticated  method such as a splitting,  trigonometric, pseudospectral, or moduled Fourier expansion method (see for instance \cite{hairer2013, gauckler2010, gauckler2015, cohen2013, faou2011, hansen2009}). The use of the spatial discretization by the MFD method combined with such symplectic integrators for the time discretization, could be the subject of future investigations for problems in two or three space dimensions.

\section{Acknowledgements}
The author Giuseppe Vacca wishes to thank the National Group of Scientific Computing (GNCS-INDAM) that through the project "Finanziamento Giovani Ricercatori 2016" has supported this research.
This paper has been partially supported by GNCS--INDAM.

\addcontentsline{toc}{section}{\refname}
\bibliographystyle{plain}
\bibliography{biblio}
\end{document}